\documentclass[12pt]{article}
\usepackage{amsmath,bm,amssymb}
\usepackage{amsthm}
\usepackage{graphicx,psfrag,epsf}
\usepackage{enumerate}
\usepackage{natbib}
\usepackage{amsfonts}
\usepackage{float}
\usepackage[T1]{fontenc}
\usepackage[utf8]{inputenc}
\usepackage{authblk}
\usepackage{epstopdf}
\usepackage{multirow}
\usepackage[table]{xcolor}
\usepackage{subfigure}

\usepackage{footnote}

\usepackage[margin=1in]{geometry}

\newtheorem{theorem}{Theorem}[section]

\newtheorem{lemma}{Lemma}[section]
\newtheorem{defn}{Definition}[section]

\newtheorem{proposition}{Proposition}[section]
\newtheorem{corollary}{Corollary}[section]

\newtheorem{expl}{Example}
\newtheorem{remark}{Remark}
\newtheorem{condition}{Condition}

\theoremstyle{definition}

\newcommand{\RR}{\mathbb{R}}

\makeatletter
\renewcommand*{\thanks}[1]{%
  \footnotemark
  \protected@xdef\@thanks{\@thanks
    \protect\footnotetext[\arabic{footnote}]{#1}}%
}
\makeatother

\begin{document}

\title{Kriging Prediction with Isotropic Mat\'ern Correlations: Robustness and Experimental Designs}

\author{Rui Tuo\thanks{Two authors  contributed equally to this work.} \thanks{Tuo's work is supported by NSF grant DMS 1914636.} \\ Department of Industrial and Systems Engineering\\
       Texas A\&M University\\
       \and Wenjia Wang$^*$\\Hong Kong University of Science and Technology\\}

\maketitle

\begin{abstract}
This work investigates the prediction performance of the kriging predictors. We derive some error bounds for the prediction error in terms of non-asymptotic probability under the uniform metric and $L_p$ metrics when the spectral densities of both the true and the imposed correlation functions decay algebraically. The Mat\'ern family is a prominent class of correlation functions of this kind. Our analysis shows that, when the smoothness of the imposed correlation function exceeds that of the true correlation function, the prediction error becomes more sensitive to the space-filling property of the design points. In particular, the kriging predictor can still reach the optimal rate of convergence, if the experimental design scheme is quasi-uniform. Lower bounds of the kriging prediction error are also derived under the uniform metric and $L_p$ metrics. An accurate characterization of this error is obtained, when an oversmoothed correlation function and a space-filling design is used.
\end{abstract}

\textit{keywords}:
  Computer Experiments, Uncertainty Quantification, Scattered Data Approximation, Space-filling Designs, Bayesian Machine Learning

\section{Introduction}

In contemporary mathematical modeling and data analysis, we often face the challenge of reconstructing smooth functions from scattered observations. Gaussian process regression, also known as kriging, is a widely used approach.
The main idea of kriging is to model the underlying function as a realization of a Gaussian process. This probabilistic model assumption endows the reconstructed function with a random distribution. Therefore, unlike the usual interpolation methods, kriging enables uncertainty quantification of the underlying function in terms of its posterior distribution given the data. In spatial statistics and engineering, Gaussian processes are used to reflect the intrinsic randomness of the underlying functions or surfaces \citep{cressie1992statistics,stein2012interpolation,matheron1963principles}. In computer experiments, the Gaussian process models are adopted so that the prediction error under limited input data can be accessed \citep{santner2013design,sacks1989design,bayarri2012framework}. For similar reasons, Gaussian process regression is applied in machine learning \citep{rasmussen2006gaussian} and probabilistic numerics \citep{hennig2015probabilistic}; specifically, in the area of Bayesian optimization, Gaussian process models are imposed and the probabilistic error of the reconstructed function are used to determine the next input point in a sequential optimization scheme for a complex black-box function \citep{shahriari2016taking,frazier2018tutorial,bull2011convergence,klein2016fast}.

Under a Gaussian process model, the conditional distribution of the function value at an untried point given the data is normal, and can be expressed explicitly. In practice, we usually use the curve of conditional expectation as a surrogate model of the underlying function. Despite the known pointwise distributions, many basic properties of the kriging predictive curves remain as open problems. In this work, we focus on three fundamental aspects of kriging: 1) convergence of kriging predictive curves in function spaces; 2) robustness of kriging prediction against misspecification of the correlation functions; 3) effects of the design of experiments. Understanding the above properties of kriging can provide guidelines for choosing suitable correlation functions and experimental designs, which would potentially help the practical use of the method.

In this article, we focus on the isotropic Mat\'ern correlation family. We suppose the underlying function is a random realization of a Gaussian process with an isotropic Mat\'ern correlation function, and we reconstruct this function using kriging with a \textit{misspecified} isotropic Mat\'ern correlation function. We summarize our main results in Section \ref{Sec:results}. In Section \ref{Sec:related}, we make some remarks on related areas and research problems, and discuss the differences between the existing and the present results. In Section \ref{Sec:background}, we state our problem formulation and discuss the required technical conditions. Our main results are presented in Section \ref{sec:mainresults}. A simulation study is reported in Section \ref{sec:simulation}, which assesses our theoretical findings regarding the effects of the experimental designs. Technical proofs are given in Section \ref{sec:proof}.

\subsection{Summary of our results}\label{Sec:results}
We consider the reconstruction of a sample path of a Gaussian process over a compact set $\Omega\subset \mathbb{R}^d$. The shape of $\Omega$ can be rather general, subject to a few regularity conditions presented in Section \ref{Sec:conditions}. Table \ref{tab:results} shows a list of results on the rate of convergence of Gaussian process regression in the $L_p(\Omega)$ norm, with $1\leq p\leq \infty$ under different designs and misspeficied correlation functions. Table \ref{tab:results} covers results on both the upper bounds and the lower bounds. The lower bounds are given in terms of the sample size $n$ and the true smoothness $\nu_0$; and the upper bounds depend also on the imposed smoothness $\nu$, and two space-filling metrics of the design: the fill distance $h_{X,\Omega}$ and the mesh ratio $\rho_{X,\Omega}$. Details of the above notation are described in Section \ref{Sec:conditions}. The variance of the (stationary) Gaussian process at each point is denoted as $\sigma^2$. Recall that we consider interpolation of Gaussian processes only, so there is no extra random error at each observed point given the Gaussian process sample path.

All results in Table \ref{tab:results} are obtained by the present work, except the shaded row which was obtained by our previous work \citep{wang2019prediction}.
Compared to \cite{wang2019prediction}, this work makes significant advances. First, this work establishes the convergence results when an oversmoothed correlation function is used, i.e.,  $\nu>\nu_0$. Specifically, the results in \cite{wang2019prediction} depends only on $h_{X,\Omega}$, and cannot be extended to oversmoothed correlations. In this work, we prove some new approximation results for radial basis functions (see Section \ref{sec:escape}), and establish the theoretical framework for oversmoothed correlations. In the present theory, the upper bounds in oversmoothed cases depend on both $h_{X,\Omega}$ and $\rho_{X,\Omega}$. We also present the bounds under the $L_p(\Omega)$ norms with $1\leq p<\infty$ as well as the lower-bound-type results in this article.

Our findings in Table \ref{tab:results} lead to a remarkable result for the so-called \textit{quasi-uniform sampling} (see Section \ref{Sec:conditions}). We show that under quasi-uniform sampling and oversmoothed correlation functions, the lower and upper rates coincide, which means that the optimal rates are achievable. This result also implies that the prediction performance does not deteriorate largely as an oversmoothed correlation function is imposed, provided that the experimental design scheme is quasi-uniform.

\begin{table}[!h]
\centering
\begin{tabular}{|c|c|c|c|}
\hline
\multicolumn{2}{|c|}{\multirow{2}{*}{Case}} &  \multicolumn{2}{|c|}{Design} \\
\cline{3-4}
       \multicolumn{2}{|c|}{}          & General design & Quasi-uniform design \\
\hline
$\nu\leq \nu_0$, & Upper rate  &  $\sigma h^\nu_{X,\Omega}$ &  $ \sigma n^{-\nu/d}$ \\
\cline{2-4}
$1\leq p<\infty$ & Lower rate  & \multicolumn{2}{|c|}{$\sigma n^{-\nu_0/d}$} \\
\hline
$\nu\leq \nu_0$, & \cellcolor[gray]{0.9}Upper rate   &\cellcolor[gray]{0.9} $\sigma h^\nu_{X,\Omega}\log^{1/2}(1/h_{X,\Omega})$ &  \cellcolor[gray]{0.9}$ \sigma n^{-\nu/d}\sqrt{\log n}$ \\
\cline{2-4}
$p=\infty$ & Lower rate  &  \multicolumn{2}{|c|}{$\sigma n^{-\nu_0/d}\sqrt{\log n}$}  \\
\hline
$\nu > \nu_0$, & Upper rate   &  $\sigma h_{X,\Omega}^{\nu_0}\rho^{\nu-\nu_0}_{X,\Omega}$ & $\sigma n^{-\nu_0/d}$ \\
\cline{2-4}
$1\leq p<\infty$ & Lower rate   & \multicolumn{2}{|c|}{$\sigma n^{-\nu_0/d}$} \\
\hline
$\nu > \nu_0$, & Upper rate   & $\sigma h_{X,\Omega}^{\nu_0}\rho^{\nu-\nu_0}_{X,\Omega} \log^{1/2}(1/h_{X,\Omega})$  & $\sigma n^{-\nu_0/d}\sqrt{\log n}$ \\
\cline{2-4}
$p=\infty$ & Lower rate  & \multicolumn{2}{|c|}{$\sigma n^{-\nu_0/d}\sqrt{\log n}$} \\
\hline
\end{tabular}
\caption{{\rm Summary of the $L_p$ convergence rates for kriging prediction error with isotropic Mat\'ern correlation functions. In addition to the rates of convergence, all kriging prediction errors in Table \ref{tab:results} decay at sub-Gaussian rates. The rates on the shaded row were presented in our previous work \citep{wang2019prediction}. The results for all other cases are obtained in the current work.} }\label{tab:results}
\end{table}

\subsection{Comparison with related areas}\label{Sec:related}
Although the general context of function reconstruction is of interest in a broad range of areas, the particular settings of this work include: 1) \textbf{Random underlying function:} the underlying function is random and follows the law of a Gaussian process; 2) \textbf{Interpolation:} besides the Gaussian process, no random error is present, and therefore an interpolation scheme should be adopted; 3) \textbf{Misspecification:} Gaussian process regression is used to reconstruct the underlying true function, and the imposed Gaussian process may have a misspecified correlation function; 4) \textbf{Scattered inputs:} the input points are fixed, with no particular structure. These features differentiate our objective from the existing areas of function reconstruction. In this section, we summarize the distinctions between the current work and four existing areas: average-case analysis of numerical problems, nonparametric regression, posterior contraction of Gaussian process priors, and scattered data approximation. Despite the differences in the scope, some of the mathematical tools in these areas are used in the present work, including a lower-bound result from the average-case analysis (Lemma \ref{th:lowLp}), and some results from the scattered data approximation (see Section \ref{sec:escape}).

\subsubsection{Average-case analysis of numerical problems}
Among the existing areas, the average-case analysis of numerical problems has the closest model settings compared with ours, where the reconstruction of Gaussian process sample paths is considered. The primary distinction between this area and our work is the objective of the study: the average-case analysis aims at finding the optimal algorithms (which are generally \textit{not} the Gaussian process regression, where a misspecified correlation can be used). In this work, we are interested in the \textit{robustness} of the Gaussian process regression. Besides, the average-case analysis focuses on the optimal designs, while our study also covers general scattered designs.

One specific topic in the average-case analysis focuses on the following quantity,
\begin{align}\label{avge}
    e_p^{\text{avg}}(\phi,N) = \bigg(\int_{F_1} \|f - \phi(Nf)\|_{L_p(\Omega)}^p\mu(df)\bigg)^{1/p},
\end{align}
where $\phi:N(F_1)\rightarrow L_p(\Omega)$ is an algorithm, $Nf = [f(x_1),...,f(x_n)]$ with $x_i\in \Omega$, and $F_1$ is a function space equipped with Gaussian measure $\mu$.
It is worth noting that the results in the present work also imply some results in the form (\ref{avge}), where $\phi$ has to be a kriging algorithm. Specifically, Theorem \ref{coro:lowLp} implies lower bounds of \eqref{avge}, and Corollary \ref{coro:quasirate} shows that these lower bounds can be achieved, which also implies upper bounds of \eqref{avge}.

\textbf{Results on the lower bounds of \eqref{avge}.} For $p=2$, the lower bound was provided by \cite{papageorgiou1990average}; also see Lemma \ref{th:lowLp}. If one further assumes that $\Omega = [0,1]^d$, Proposition VI.8 of \cite{ritter2007average} shows that the error \eqref{avge} has a lower bound with the rate $n^{-\nu_0/d}$. One dimensional problems with correlation functions satisfying the Sacks-Ylvisaker conditions are extensively studied; see \cite{muller1997uniform,ritter2007average,ritter1995multivariate,sacks1966designs,sacks1968designs,sacks1970designs}.

\textbf{Results on the upper bound of \eqref{avge}.} Upper-bound-type results are pursued in average-case analysis under the optimal algorithm $\phi$ and optimal designs of $\{x_1,...,x_n\}$. If $\Omega = [0,1]^d$, \cite{ritter2007average} shows that the rate $n^{-\nu_0/d}$ can be achieved by piecewise polynomial interpolation and specifically chosen designs; see Remark VI.3 of \cite{ritter2007average}, also see page 34 of \cite{novak2006deterministic} and \cite{vvivanov1971}.

For $1\leq p <\infty$ and the Mat\'ern correlation function in one dimension, the error in average case $e_p^{\text{avg}}(\phi,N)$ can achieve the rate $n^{-\nu_0}$ by using piecewise polynomial interpolation; See Proposition IV.36
of \cite{ritter2007average}. For the Mat\'ern correlation function in one dimension, the quantity
\begin{align}\label{avgeinfty}
    e_{L_\infty,p}^{\text{avg}}(\phi,N) = \bigg(\int_{F_1} \|f - \phi(Nf)\|_{L_\infty(\Omega)}^p\mu(df)\bigg)^{1/p},
\end{align}
can achieve the rate $n^{-\nu_0}\sqrt{\log n}$ by using Hermite interpolating splines \citep{seleznjev1999certain} for $1\leq p <\infty$.

Other definitions of the error are also studied in average-case analysis. See \cite{chen2019average,fasshauer2012average,khartov2019asymptotic,lifshits2015approximation,luschgy2004sharp,luschgy2007high} for examples.

\subsubsection{Nonparametric regression and statistical learning}
The problem of interest in nonparameteric regression is to recover a \textit{deterministic} function $f$ under $n$ \textit{noisy} observations $(x_i,y_i),i=1,\ldots,n$, under the model
\begin{eqnarray}
y_i=f(x_i)+\epsilon_i, & i=1,\ldots,n,\label{recovering}
\end{eqnarray}
where $\epsilon_i$'s are the measurement error. Assuming that the function $f$ has smoothness $\nu_0$,\footnote{See Section \ref{sec:escape} for a discussion on the smoothness of a deterministic function.} the optimal (minimax) rate of convergence is  $n^{-\nu_0/(2\nu_0 + d)}$ \citep{stone1982optimal}. A vast literature proposes and discusses methodologies regarding the nonparametric regression model \eqref{recovering}, such as smoothing splines \citep{gu2013smoothing}, kernel ridge regression \citep{geer2000empirical}, local polynomials \citep{Tsybakov2008}, etc. Because of the random noise, the rates for nonparametric regression are slower than those of the present work, as well as those in other interpolation problems.
Some cross-cutting theory and approaches between regression and scattered data approximation are also discussed in the statistical learning literature; see, for example, \cite{cucker2007learning}.

\subsubsection{Posterior contraction of Gaussian process priors}
In this area, the model setting is similar to nonparametric regression, i.e., the underlying function is assumed \textit{deterministic} and the observations are subject to random \textit{noise}. The problem of interest is the posterior contraction of the Gaussian process prior. An incomplete list of papers in this area includes \cite{castillo2008lower,castillo2014bayesian,giordano2019consistency,nickl2017nonparametric,pati2015optimal,vaart2011information,van2008rates,van2016gaussian}. Despite the use of Gaussian process priors, to the best of our knowledge, the theory under this framework does not consider noiseless observations, and no error bounds in terms of the our settings, i.e., the fill and separation distances, are reported in this area.

\subsubsection{Scattered data approximation}\label{sec:SDA}
In the field of scattered data approximation, the goal is to approximate, or often, interpolate a \textit{deterministic} function $f$ with its exact observations $f(x_i),i=1,...,n$, where $x_i$'s are data sites. For function $f$ with smoothness $m$, the $L_p$ convergence rate is $n^{-m/d+(1/2-1/q)_+}$ for $1\leq p\leq \infty$ \citep{wendland2004scattered}, where $a_+$ stands for $\max\{a,0\}$. A sharper characterization of the upper bounds are related to the fill distance and separation distance of the design points. Although this area focuses on a purely deterministic problem, some of the results in this field will serve as the key mathematical tool in this work.

It is worth noting that the existing research in scattered data approximation also covered the circumstances where the underlying function is rougher than the kernel function, so that the function is outside of the reproducing kernel Hilbert space generated by the kernel. See \cite{narcowich2006sobolev} for example. Such results can be interpreted as using ``misspecified'' kernels in interpolating deterministic functions. More discussions are deferred to Section \ref{sec:escape}.

\section{Problem formulation}\label{Sec:background}

In this section we discuss the interpolation method considered in this work, and the required technical conditions.

\subsection{Background}\label{Sec:backgroundsub}
Let $Z(x)$ be an underlying Gaussian process, with $x\in\mathbb{R}^d$.
We suppose $Z(\cdot)$ is a stationary Gaussian process with mean zero. The covariance function of $Z$ is denoted as
\begin{eqnarray*}
\text{Cov}(Z(x),Z(x'))=\sigma^2 \Psi(x-x'),
\end{eqnarray*}
where $\sigma^2$ is the variance, and $\Psi$ is the correlation function, or kernel, satisfying $\Psi(0)=1$. The correlation function $\Psi$ is a symmetric positive semi-definite function on $\mathbb{R}^d$. Since we are interested in interpolation, we require that $Z(\cdot)$ is mean square continuous, or equivalently, $\Psi$ is continuous on $\mathbb{R}^d$. Then it follows from the Bochner's theorem (\citealp[page 208]{gihman1974theory};   \citealp[Theorem 6.6]{wendland2004scattered}) that, there exists a finite nonnegative Borel measure $F_\Psi$ on $\mathbb{R}^d$, such that
\begin{eqnarray}\label{Bochner}
\Psi(x)=\int_{\mathbb{R}^d} e^{i \omega^T x}  F_\Psi(d \omega).
\end{eqnarray}
In particular, we are interested in the case where $\Psi$ is also positive definite and integrable on $\mathbb{R}^d$. In this case, it can be proven that $F_\Psi$ has a density with respect to the Lebesgue measure. See Theorem 6.11 of \cite{wendland2004scattered}. The density of $F_\Psi$, denoted as $f_\Psi$, is known as the \textit{spectral density} of $Z$ or $\Psi$.

In this work, we suppose that $f_\Psi$ decays algebraically. A prominent class of correlation functions of this type is the isotropic Mat\'ern correlation family \citep{santner2013design,stein2012interpolation}, given by \begin{eqnarray}\label{matern}
\Psi( x;\nu,\phi)=
\frac{1}{\Gamma(\nu)2^{\nu-1}}(2\sqrt{\nu}\phi \| x\|)^\nu K_\nu(2\sqrt{\nu}\phi\| x\|),
\end{eqnarray}
with the spectral density \citep{tuo2015theoretical}
\begin{eqnarray}\label{maternspectral}
f_\Psi(\omega;\nu,\phi)= \pi^{-d/2}\frac{\Gamma(\nu+d/2)}{\Gamma(\nu)}(4\nu \phi^2)^\nu (4\nu\phi^2+\|{\omega}\|^2)^{-(\nu+d/2)},
\end{eqnarray}
where $\phi,\nu>0$, $K_\nu$ is the modified Bessel function of the second kind and $\|\cdot\|$ denotes the Euclidean metric. It is worth noting that (\ref{maternspectral}) is bounded above and below by $(1+\|\omega\|^2)^{-(\nu+d/2)}$ multiplied by two constants, respectively. The parameter $\nu$ for the Mat\'ern kernels is called the \textit{smoothness} parameter, as it governs the smoothness (or differentiability) of the Gaussian processes. Further discussions are deferred to Section \ref{sec:escape}.

Another example of correlation functions with algebraically decayed spectral densities is the generalized Wendland correlation function \citep{wendland2004scattered,gneitingstationary,chernih2014closed, bevilacqua2019estimation,fasshauer2015kernel}, defined as
\begin{align*}
    \Psi_{GW}(x) = \left\{
    \begin{array}{lc}
         \frac{1}{B(2\kappa,\mu+1)}\int_{\|\phi x\|}^1 u(u^2-\|\phi x\|^2)^{\kappa - 1}(1-u)^\mu du, & 0\leq \|x\|<\frac{1}{\phi}, \\
         0, & \|x\|\geq \frac{1}{\phi},
    \end{array}\right.
\end{align*}
where $\phi,\kappa > 0$ and $\mu \geq (d+1)/2 + \kappa$, and $B$ denotes the beta function. See Theorem 1 of \cite{bevilacqua2019estimation}.

Now we consider the interpolation problem. Suppose we have a scattered set of points $X=\{x_1,\ldots,x_n\}\subset\Omega$. Here the set $\Omega$ is the region of interest, which is a subset of $\mathbb{R}^d$. The goal of kriging is to recover $Z(x)$ given the observed data $Z(x_1),\ldots,Z(x_n)$. A standard predictor is the best linear predictor \citep{santner2013design,stein2012interpolation}, given by the conditional expectation of $Z(x)$ on $Z(x_1),\ldots,Z(x_n)$, as
\begin{eqnarray}\label{BLP}
\mathbb{E}[Z(x)|Z(x_1),\ldots,Z(x_n)]=r_\Psi^T(x) K_\Psi^{-1} Y,
\end{eqnarray}
where $r_\Psi(x)=(\Psi(x-x_1),\ldots,\Psi(x-x_n))^T, K_\Psi=(\Psi(x_j-x_k))_{j k}$ and $Y=(Z(x_1),\ldots,$ $Z(x_n))^T$.

The best linear predictor in (\ref{BLP}) depends on the correlation function $\Psi$. However, in practice $\Psi$ is commonly unknown. Thus, we may inevitably use a misspecified correlation function, denoted by $\Phi$. Suppose that $\Phi$ has a spectral density $f_\Phi$.
We also suppose that $f_\Phi$ decays algebraically, but the decay rate of $f_\Phi$ can differ from that of $f_\Psi$.

We consider the predictor given by the right-hand side of (\ref{BLP}), in which the true correlation function $\Psi$ is replaced by the misspecified correlation function $\Phi$.
Clearly, such a predictor is no longer the best linear predictor. Nevertheless, it still defines an interpolant, denoted by
\begin{eqnarray}\label{interpolant}
\mathcal{I}_{\Phi,X}Z(x)=r^T_\Phi(x) K_\Phi^{-1} Y,
\end{eqnarray}
where $r_\Phi(x)=(\Phi(x-x_1),\ldots,\Phi(x-x_n))^T, K_\Phi=(\Phi(x_j-x_k))_{j k}$ and $Y=(Z(x_1),\ldots,$ $Z(x_n))^T$. In (\ref{interpolant}), $\mathcal{I}_{\Phi,X}$ denotes the interpolation operator given by the kriging predictor, which can be applied not only to a Gaussian process, but also to a deterministic function in the same vein.

\subsection{Notation and conditions} \label{Sec:conditions}
We do not assume any particular structure of the design points $X=\{x_1,\ldots,x_n\}$. Our error estimate for the kriging predictor will depend on two dispersion indices of the design points.

The first one is the fill distance, defined as
$$h_{X,\Omega}:= \sup_{x\in\Omega}\inf_{x_j\in X}\|x-x_j\|.$$
The second is the separation radius, given by
$$ q_X:=\min_{1\leq j\neq k\leq n}\|x_j-x_k\|/2. $$
It is easy to check that $h_{X,\Omega}\geq q_X$ \citep{wendland2004scattered}.
Define the mesh ratio $\rho_{X,\Omega}:=h_{X,\Omega}/q_X\geq 1$.
Because we are only interested in the prediction error when the design points are sufficiently dense, for notational simplicity, we assume that $h_{X,\Omega}<1$. In the rest of this paper, we use the following conventions. For two positive sequences $a_n$ and $b_n$, we write $a_n\asymp b_n$ if, for some constants $C,C'>0$, $C\leq a_n/b_n \leq C'$ for all $n$, and write $a_n\gtrsim b_n$ if $a_n\geq Cb_n$ for some constant $C>0$. Let card$(X)$ denote the cardinality of set $X$.

In this work, we consider both the non-asymptotic case, i.e., the design $X$ is fixed, and the asymptotic case, i.e., the number of design points increases to infinity. To state the asymptotic results, suppose we have a sequence of designs with increasing number of points, denoted by $\mathcal{X}=\{X_1,X_2,\ldots\}$. We regard $\mathcal{X}$ as a \textit{sampling scheme} which generates a sequence of designs, for instance, a design sequence generated by random sampling or maximin Latin hypercube designs.

Without loss of generality, assume that ${\rm card}(X_n)=n$, where $n$ takes its value in an infinite subset of $\mathbb{N}$. This assumption enables direct comparison between our upper and lower bounds. Given the sampling scheme $\mathcal{X}$, we denote $h_n:=h_{X_n,\Omega},q_n:=q_{X_n}$ and $\rho_n=h_n/q_n$. For any sampling scheme, it can be shown that $h_n \gtrsim  n^{-1/d}$ and $q_n\lesssim n^{-1/d}$  \citep{borodachov2007asymptotics,joseph2015sequential}. In fact, it is possible to have $h_n\asymp q_n\asymp n^{-1/d}$, if and only if $\rho_n$ is uniformly bounded above by a constant \citep{muller2009komplexitat}.

\begin{defn}
A sampling scheme $\mathcal{X}$ is said quasi-uniform if there exists a constant $C>0$ such that $\rho_n\leq C$ for all $n$.
\end{defn}

It is not hard to find a quasi-uniform sampling scheme. For example,
a hypercube grid sampling in $\Omega=[0,1]^d$ is quasi-uniform because $\rho_n=\sqrt{d}$ is a constant \citep{wendland2004scattered}. However, random samplings do not belong to the quasi-uniform class; see Example \ref{examplerdsample} in Section \ref{sec:example}.

We assume the Conditions \ref{C4}-\ref{C3} throughout this article.

\begin{defn}\label{def:icc}
	A set $\Omega\subset{\mathbb{R}^d}$ is said to satisfy an interior cone condition if there exists an angle $\alpha\in (0,\pi/2)$ and a radius $R>0$ such that for every $x\in\Omega$, a unit vector $\xi(x)$ exists such that the cone
	$$C(x,\xi(x),\alpha,R):=\left\{x+\lambda y:y\in\mathbb{R}^d,\|y\|=1,y^T\xi(x)\geq \cos\alpha,\lambda\in[0,R]\right\} $$
	is contained in $\Omega$.
\end{defn}

\begin{condition}\label{C4}
	The experimental region	$\Omega\subset \mathbb{R}^d$ is a compact set with Lipschitz boundary and satisfies an interior cone condition.
\end{condition}

\begin{condition}\label{C1}
	There exist constants $c_2 \geq c_1>0$ and $\nu_0>0$ such that, for all $\omega\in\mathbb{R}^d$,
	$$ c_1(1+\|\omega\|^2)^{-(\nu_0+d/2)} \leq  f_\Psi(\omega)\leq c_2(1+\|\omega\|^2)^{-(\nu_0+d/2)}. $$
	\end{condition}

\begin{condition}\label{C3}
	There exist constants $c_4\geq c_3>0$ and $\nu>0$ such that, for all $\omega\in\mathbb{R}^d$,
	$$ c_3(1+\|\omega\|^2)^{-(\nu+d/2)}\leq f_\Phi(\omega)\leq c_4(1+\|\omega\|^2)^{-(\nu+d/2)}. $$
\end{condition}

Condition \ref{C4} is a geometric condition on the experimental region $\Omega$, which holds in most practical situations, because the commonly encountered experimental regions, like the rectangles or balls, satisfy interior cone conditions. Figure \ref{figint} (page 258 of \cite{roy2001inverse}) is an illustration of the $\alpha$-interior cone condition.

\begin{figure}[h!]
    \centering
    \includegraphics[height=2in]{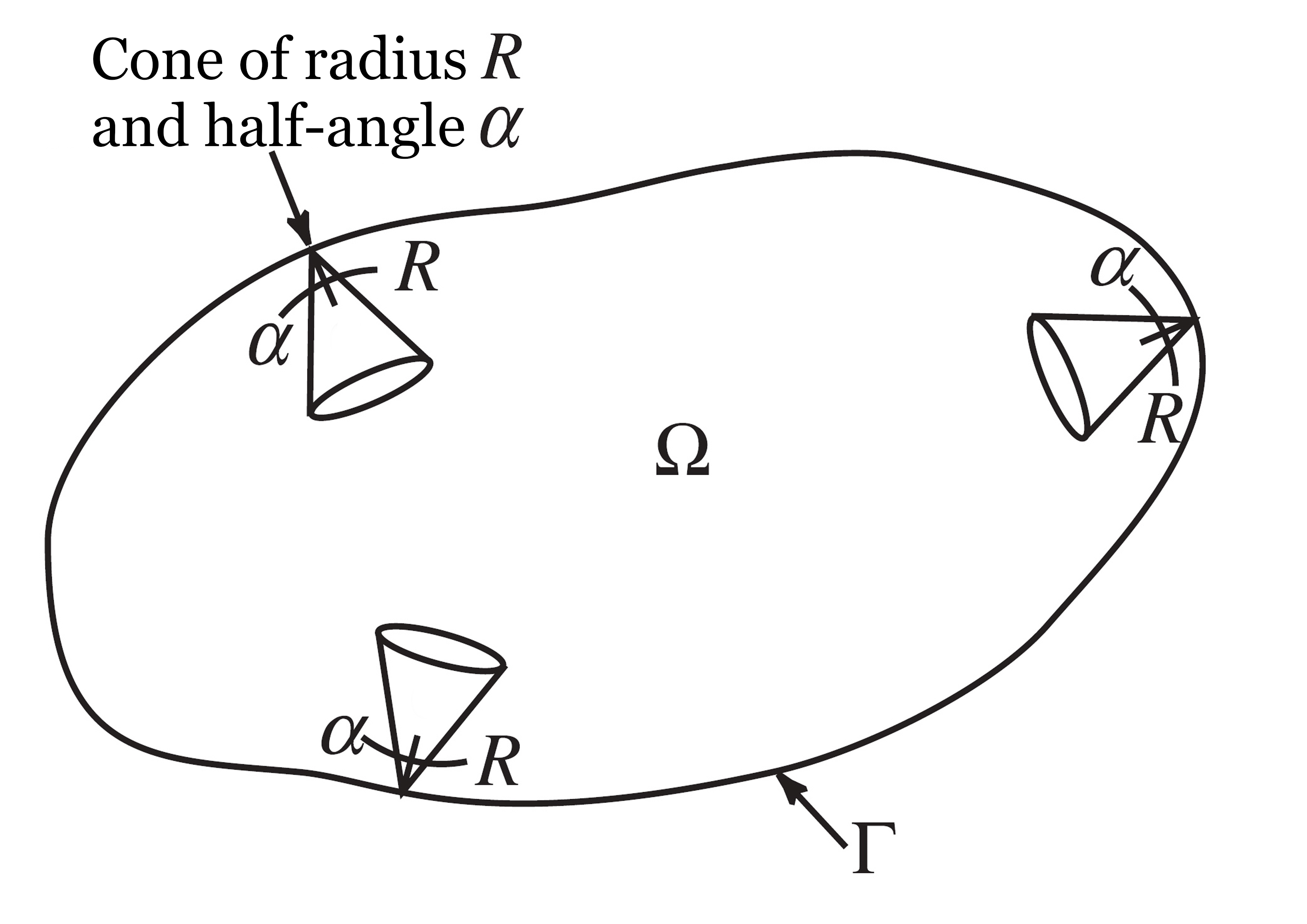}
    \caption{{\rm An illustration of an interior cone condition (page 258 of \cite{roy2001inverse}).}}\label{figint}
\end{figure}

Conditions \ref{C1} and \ref{C3} require that the spectral densities decay in an algebraic order. For example, if $\Psi$ and $\Phi$ are Mat\'ern correlation functions with smoothness parameter $\nu_0$ and $\nu$, respectively, they satisfy Conditions \ref{C1} and \ref{C3}. The decay rates in Conditions \ref{C1} and \ref{C3} determine the smoothness of the correlation function $\Psi$ and $\Phi$; see Section \ref{sec:escape} for the discussion of the relation between the smoothness of the correlation functions and the smoothness of the sample path of a Gaussian process.

\section{Main results}\label{sec:mainresults} In this section, we present our main theoretical results on the prediction error of kriging.

\subsection{Upper and lower bounds of the uniform kriging prediction error}\label{sec:mainLinfty}

This work aims at studying the prediction error of the kriging algorithm (\ref{interpolant}), i.e., $|Z(x)-\mathcal{I}_{\Phi,X}Z(x)|$. In this subsection, we consider the prediction error of the kriging algorithm (\ref{interpolant}) under a uniform metric, given by
\begin{eqnarray}\label{krigingerror}
\sup_{x\in\Omega}|Z(x)-\mathcal{I}_{\Phi,X}Z(x)|,
\end{eqnarray}
which was considered previously in \cite{wang2019prediction}. Under Conditions \ref{C4}-\ref{C3}, they derived an upper bound of (\ref{krigingerror}) under the case $\nu\leq \nu_0$. This result is shown in Theorem \ref{Th:wtw} for the completeness of work. Here we are interested in the case $\nu>\nu_0$, that is, the imposed correlation function is smoother than the true correlation function. In Theorem \ref{Th:main}, we provide an upper bound of the prediction error for $\nu>\nu_0$. In addition to the upper bounds, we obtain a lower bound of the uniform kriging prediction error in Theorem \ref{th:low}.

\begin{theorem}\label{Th:wtw}
Suppose Conditions \ref{C4}-\ref{C3} hold and $\nu\leq\nu_0$. Then there exist constants $C_1,C_2>0, C_3>e$ and $h_0\in(0,1]$, such that for any design $X$ with $h_{X,\Omega}\leq h_0$ and any $t>0$, with probability at least $1-\exp\{-t^2/(C_1\sigma^2h^{2\nu}_{X,\Omega})\}$,\footnote{In \cite{wang2019prediction}, this probability is $1-2\exp\{-t^2/(C_1\sigma^2h^{2\nu}_{X,\Omega})\}$. The constant two can be removed by applying a different version of the Borell-TIS inequality given by Lemma \ref{concentrationGP} in Section \ref{sec:proofLinfty}.} the kriging prediction error has the upper bound
\begin{eqnarray*}
	\sup_{x\in\Omega}|Z(x)-\mathcal{I}_{\Phi,X}Z(x)|\leq C_2\sigma h^\nu_{X,\Omega}\log^{1/2}\left(C_3/h_{X,\Omega}\right)+t.
\end{eqnarray*}
Here the constants $C_1,C_2,C_3$ depend only on $\Omega,\Phi$, and $\Psi$, including $\nu$ and $\nu_0$.
\end{theorem}

\begin{theorem}\label{Th:main}
	Suppose Conditions \ref{C4}-\ref{C3} hold and $\nu>\nu_0$. Then there exist constants $C_1,C_2>0, C_3>e$ and $h_0\in(0,1]$, such that for any design $X$ with $h_{X,\Omega}\leq h_0$ and any $t>0$, with probability at least $1-\exp\{-t^2/(C_1\sigma^2 h_{X,\Omega}^{2\nu_0}\rho^{2(\nu-\nu_0)}_{X,\Omega})\}$, the kriging prediction error has the upper bound
	$$ \sup_{x\in\Omega}|Z(x)-\mathcal{I}_{\Phi,X}Z(x)|\leq C_2 \sigma h_{X,\Omega}^{\nu_0}\rho^{\nu-\nu_0}_{X,\Omega} \log^{1/2}(C_3/h_{X,\Omega})+t. $$
	Here the constants $C_1,C_2,C_3$ depend only on $\Omega,\Phi$, and $\Psi$, including $\nu$ and $\nu_0$.
\end{theorem}

\begin{theorem}\label{th:low}
Suppose Conditions \ref{C4}-\ref{C3} hold. Then there exist constants $C_1,C_2>0$, such that for any design $X$ satisfying ${\rm card}(X)=n$ and any $t>0$, with probability at least $1-\exp\{-t^2/(2C_1\sigma^2 A)\}$, the kriging prediction error has the lower bound
\begin{eqnarray*}
\sup_{ x\in\Omega}|Z( x)-\mathcal{I}_{\Phi,X}Z( x)|\geq C_2\sigma n^{-\frac{\nu_0}{d}}\sqrt{\log n} - t,
\end{eqnarray*}
where $A = h^{2\nu}_{X,\Omega}$ if $\nu \leq \nu_0$, and $A =h_{X,\Omega}^{2\nu_0}\rho^{2(\nu-\nu_0)}_{X,\Omega}$ if $\nu > \nu_0$. Here the constants $C_1,C_2>0$ depend only on $\Omega,\Phi$, and $\Psi$, including $\nu$ and $\nu_0$.
\end{theorem}

\subsection{Bounds for the $L_p$ norms of the kriging prediction error} \label{sec:bounds}
Now we consider the $L_p$ norm of the kriging prediction error, given by
\begin{eqnarray}\label{krigingerrorLp}
\|Z-\mathcal{I}_{\Phi,X}Z\|_{L_p(\Omega)}:=\left(\int_\Omega |Z(x)-\mathcal{I}_{\Phi,X}Z(x)|^p d x\right)^{1/p},
\end{eqnarray}
with $1\leq p < \infty$. The upper bounds of the $L_p$ norms of the kriging prediction error with undersmoothed and oversmoothed correlation functions are provided in Theorems \ref{Th:wtwLp} and \ref{Th:mainLp}, respectively.

\begin{theorem}\label{Th:wtwLp}
Suppose Conditions \ref{C4}-\ref{C3} hold and $\nu\leq\nu_0$. Then there exist constants $C_1,C_2>0$ and $h_0\in(0,1]$, such that for any design $X$ with $h_{X,\Omega}\leq h_0$ and any $t>0$, with probability at least $1-\exp\{-t^2/(C_1\sigma^2h^{2\nu}_{X,\Omega})\}$, the kriging prediction error has the upper bound
\begin{eqnarray*}
	\|Z-\mathcal{I}_{\Phi,X}Z\|_{L_p(\Omega)}\leq C_2\sigma h^\nu_{X,\Omega}+t.
\end{eqnarray*}
The constants $C_1,C_2$ depend only on $\Omega, p, \Phi$, and $\Psi$, including $\nu$ and $\nu_0$.
\end{theorem}

\begin{theorem}\label{Th:mainLp}
	Suppose Conditions \ref{C4}-\ref{C3} hold and $\nu>\nu_0$. Then there exist constants $C_1,C_2>0$ and $h_0\in(0,1]$, such that for any design $X$ with $h_{X,\Omega}\leq h_0$ and any $t>0$, with probability at least $1-\exp\{-t^2/(C_1\sigma^2 h_{X,\Omega}^{2\nu_0}\rho^{2(\nu-\nu_0)}_{X,\Omega})\}$, the kriging prediction error has the upper bound
	$$ \|Z-\mathcal{I}_{\Phi,X}Z\|_{L_p(\Omega)}\leq C_2 \sigma h_{X,\Omega}^{\nu_0}\rho^{\nu-\nu_0}_{X,\Omega} +t. $$
	Here the constants $C_1,C_2$ depend only on $\Omega, p, \Phi$, and $\Psi$, including $\nu$ and $\nu_0$.
\end{theorem}

Regarding the lower prediction error bounds under the $L_p$ norm, we obtain a result analogous to Theorem \ref{th:low}. Theorem \ref{coro:lowLp} suggests a lower bound under the $L_p$ norm, which differs from that in Theorem \ref{th:low} only by a $\sqrt{\log n}$ factor.

\begin{theorem}\label{coro:lowLp}
Suppose Conditions \ref{C4}-\ref{C3} hold. There exist constants $C_1,C_2>0$, such that for any design $X$ satisfying ${\rm card}(X)=n$ and any $t>0$, with probability at least $1-2\exp\{-t^2/(2C_1\sigma^2 A)\}$, the kriging prediction error has the lower bound
\begin{eqnarray*}
\|Z-\mathcal{I}_{\Phi,X}Z\|_{L_p(\Omega)} \geq C_2\sigma n^{-\frac{\nu_0}{d}}- t
\end{eqnarray*}
for $1 \leq p< \infty$, where $A = h^{2\nu}_{X,\Omega}$ if $\nu \leq \nu_0$, and $A =h_{X,\Omega}^{2\nu_0}\rho^{2(\nu-\nu_0)}_{X,\Omega}$ if $\nu > \nu_0$. Here the constants $C_1,C_2>0$ depend only on $\Omega, p, \Phi$, and $\Psi$, including $\nu$ and $\nu_0$.
\end{theorem}

The results in Theorems \ref{Th:wtw}, \ref{Th:main}, \ref{Th:wtwLp} and \ref{Th:mainLp} are presented in a non-asymptotic manner, i.e., the design $X$ is fixed. The asymptotic results, which are traditionally of interest in spatial statistics, can be inferred from these non-asymptotic results. Here we consider the so-called fixed-domain asymptotics \citep{stein2012interpolation,loh2005fixed}, in which the domain $\Omega$ is kept unchanged and the design points become dense over $\Omega$.

We collect the asymptotic rates analogous to the upper bounds in Corollaries \ref{coro:rate} and \ref{coro:rateLp}. Their proofs are straightforward.

\begin{corollary}\label{coro:rate}
	Suppose Conditions \ref{C4}-\ref{C3} hold. In addition, we suppose the sampling scheme $\mathcal{X}$ is asymptotically dense over $\Omega$, that is, $h_n\rightarrow 0$ as $n\rightarrow \infty$. We further assume $h_n^{\nu_0}\rho^{(\nu-\nu_0)}_n\rightarrow 0$ if $\nu>\nu_0$. Then the uniform kriging prediction error has the order of magnitude
	\begin{align*}
		\sup_{ x\in\Omega}|Z(x)-\mathcal{I}_{\Phi,X_n}Z(x)| =
	\begin{cases}
	O_{\mathbb{P}}\left(h_n^\nu\log^{1/2}(1/h_n)\right) &\text{if } \nu\leq \nu_0,\\
	O_{\mathbb{P}}\left(h_n^{\nu_0}\rho_n^{\nu-\nu_0}\log^{1/2}(1/h_n)\right)&\text{if } \nu>\nu_0.
	\end{cases}
	\end{align*}
\end{corollary}

\begin{corollary}\label{coro:rateLp}
	Under the conditions of Corollary \ref{coro:rate}, for $1\leq p<\infty$, the kriging prediction error has the order of magnitude in $L_p(\Omega)$
	\begin{align*}
		\|Z(x)-\mathcal{I}_{\Phi,X_n}Z(x)\|_{L_p(\Omega)} =
	\begin{cases}
	O_{\mathbb{P}}\left(h_n^\nu\right) &\text{if } \nu\leq \nu_0,\\
	O_{\mathbb{P}}\left(h_n^{\nu_0}\rho_n^{\nu-\nu_0}\right)&\text{if } \nu>\nu_0.
	\end{cases}
	\end{align*}
\end{corollary}

From Corollaries \ref{coro:rate} and \ref{coro:rateLp}, we find that the upper bounds of kriging prediction error strongly depend on the sampling scheme $\mathcal{X}$.

If a sampling scheme is quasi-uniform and $\nu\geq \nu_0$, then the orders of magnitude in Corollaries \ref{coro:rate} and \ref{coro:rateLp} agree with the lower bounds in Theorems \ref{th:low} and \ref{coro:lowLp}, respectively, implying that these bounds are sharp. We summarize the results in Corollary \ref{coro:quasirate}.

\begin{corollary}\label{coro:quasirate}
Suppose Conditions \ref{C4}-\ref{C3} hold and $\nu\geq \nu_0$. In addition, we suppose the sampling scheme $\mathcal{X}$ is quasi-uniform. Then the kriging prediction error has the exact order of magnitude
\begin{align*}
\sup_{ x\in\Omega}|Z(x)-\mathcal{I}_{\Phi,X_n}Z(x)| \asymp &  n^{-\nu_0/d}\log^{1/2} n,\\
\|Z(x)-\mathcal{I}_{\Phi,X_n}Z(x)\|_{L_p(\Omega)} \asymp &  n^{-\nu_0/d}, ~~~ 1\leq p < \infty.
\end{align*}
\end{corollary}

\subsection{An example}\label{sec:example}
We illustrate the impact of the experimental designs in Example \ref{examplerdsample}.

\begin{expl}\label{examplerdsample}
	The random sampling in $[0,1]$ is \textit{not} quasi-uniform. To see this, let $x_1,\ldots,x_n$ be mutually independent random variables following the uniform distribution on $[0,1]$. Denote their order statistics as
	$$0=x_{(0)}\leq x_{(1)}\leq\cdots\leq x_{(n)}\leq x_{(n+1)}=1. $$
	Clearly, we have
	$$\rho_n=\frac{\max_{0\leq j\leq n}|x_{(j+1)}-x_{(j)}|}{\min_{0\leq j\leq n}|x_{(j+1)}-x_{(j)}|}.$$
	
Let $y_1,\ldots,y_n,y_{n+1}$ be mutually independent random variables following the exponential distribution with mean one. It is well known that $(x_{(1)},\ldots,x_{(n)})$ has the same distribution as
$$\left(\frac{y_1}{\sum_{j=1}^{n+1} y_j},\ldots,\frac{\sum_{j=1}^{n} y_j}{\sum_{j=1}^{n+1} y_j}\right). $$
Thus $\rho_n$ has the same distribution as $\max y_j/\min y_j$. Clearly, $\max y_j\asymp\log n$ and $\min y_j\asymp1/n$. This implies $\rho_n\asymp n\log n$. Similarly, we can see that $h_n$ has the same distribution as $\max y_j/\sum_{k=1}^{n+1} y_k$, which is of the order $O_{\mathbb{P}}(n^{-1}\log n)$. See Appendix \ref{app:distribution} for proofs of the above statements.

Now consider the kriging predictive curve under $\Omega=[0,1]$ and random sampled design points and an oversmoothed correlation, i.e., $\nu>\nu_0$.
According to Corollary \ref{coro:rate}, its uniform error has the order of magnitude $O_{\mathbb{P}}(n^{\nu-2\nu_0}\log^{\nu+1/2} n)$, which decays to zero if $\nu<2\nu_0$.

In Section \ref{sec:simulation}, we will conduct simulation studies to verify our theoretical assertions on the rates of convergence in this example. It can be seen from Table \ref{Tab:simuResults} in Section \ref{sec:simulation} that the numerical results agree with our theory.
\end{expl}

\section{Discussion on a major mathematical tool and the notion of smoothness}
\label{sec:escape}
The theory of radial basis function approximation is an essential mathematical tool for developing the bounds in this work, as well as those in our previous work \cite{wang2019prediction}.
We refer to \cite{wendland2004scattered} for an introduction of the radial basis function approximation theory.

A primary objective of the radial basis function approximation theory is to study the approximation error
$$g-\mathcal{I}_{\Phi,X}g, $$
for a deterministic function $g$. Here we consider the circumstance that $g$ lies in a (fractional) Sobolev space.

Our convention of the Fourier transform is
$\hat{g}(\omega)=\int_{\mathbb{R}^d} g(x)e^{-i\omega^Tx} dx.$
Regarding the Fourier transform as a mapping $ \hat{g}: L_1(\mathbb{R}^d)\cap L_2(\mathbb{R}^d)\rightarrow L_2(\mathbb{R}^d)$, we can uniquely extend it to a mapping $\hat{g}: L_2(\mathbb{R}^d)\rightarrow L_2(\mathbb{R}^d)$ \citep{wendland2004scattered}.
The norm of the (fractional) Sobolev space  $W_2^\beta(\mathbb{R}^d)$ for a real number $\beta>0$ (also known as the Bessel potential space) is
\begin{eqnarray*}
	\|g\|^2_{W^\beta_2(\mathbb{R}^d)}=\int_{\mathbb{R}^d} |\hat{g}(\omega)|^2(1+\|\omega\|^2)^{\beta} d \omega,
\end{eqnarray*}
for $g\in L_2(\mathbb{R}^d)$.

\begin{remark}
An equivalent norm of the Sobolev space $W_2^\beta(\mathbb{R}^d)$ for $\beta\in \mathbb{N}$ can be defined via derivatives. For $\alpha=(\alpha_1,\ldots,\alpha_d)^T\in\mathbb{N}^d_0$, we shall use the notation $|\alpha|=\sum_{j=1}^d \alpha_j$. For $x=(x_1,\ldots,x_d)^T$, denote
$$D^\alpha g=\frac{\partial^{|\alpha|} }{\partial x_1^{\alpha_1}\cdots\partial x_d^{\alpha_d}}g  \text{~~~~~~~and~~~~~~~} x^\alpha=x_1^{\alpha_1}\cdots x_d^{\alpha_d} .$$ Define $\|g\|_{W^\beta_2(\mathbb{R}^d)}'= \left(\sum_{|\alpha|\leq \beta}\|D^\alpha g\|_{L_2(\RR^d)}^2\right)^{\frac{1}{2}}$. It can be shown that $\|\cdot\|_{W^\beta_2(\mathbb{R}^d)}'$ and $\|\cdot\|_{W^\beta_2(\mathbb{R}^d)}$ are equivalent for $\beta\in \mathbb{N}$ \citep{adams2003sobolev}.
\end{remark}

The classic framework on the error analysis for radial basis function approximation employs the reproducing kernel Hilbert spaces (RKHS, see Section \ref{sec:RKHS} for more details) as a necessary mathematical tool. The development of \cite{wang2019prediction} relies on these classic results. These results, however, are not applicable in the current context when $f_\Psi/f_\Phi$ is not uniformly bounded.

The current research is partially inspired by the ``escape theorems'' for radial basis function approximation established  by \cite{brownlee2004approximation,narcowich2005sobolev,narcowich2005recent,narcowich2002scattered,narcowich2004scattered,narcowich2006sobolev}. These works show that, some radial basis functions interpolants still provide effective approximation, even if the underlying functions are too rough to lie in the corresponding RKHS.

Our results on interpolation of Gaussian processes with oversmoothed kernels are based on an escape theorem, given by Lemma \ref{th:escape}. Given Condition \ref{C3}, it is known that the RKHS generated by $\Phi$ is equivalent to $W_2^{\nu+d/2}(\mathbb{R}^d)$ (see Lemma \ref{lemEquivalence} in Section \ref{sec:RKHS}), which is a proper subset of $W_2^{\nu_0+d/2}(\mathbb{R}^d)$ when $\nu_0<\nu$. Lemma \ref{th:escape} shows that the radial basis function approximation may still give reasonable error bounds even if the underlying function does not lie in the RKHS.

\begin{lemma}\label{th:escape}
	Let $\Phi$ be a kernel with a spectral density $f_\Phi$ satisfying Condition \ref{C3},
	and $g$ be a function in $ W_2^{\nu_0+d/2}(\mathbb{R}^d)$ with $\nu\geq \nu_0>0$.
	Suppose $\Omega\subset \mathbb{R}^d$ is a domain satisfying Condition \ref{C4}. Then there exist constants $C>0$ and $h_0\in(0,1]$ such that for any design $X$ with $h_{X,\Omega}\leq h_0$, we have
	\begin{eqnarray}\label{infnorm}
	\sup_{x\in\Omega}|g(x)-\mathcal{I}_{\Phi,X}g(x)|\leq C h_{X,\Omega}^{\nu_0}\rho^{\nu-\nu_0}_{X,\Omega}\|g\|_{W_2^{\nu_0+d/2}(\mathbb{R}^d)}.
	\end{eqnarray}
	Here the constant $C$ depends only on $\Omega$, $\Phi$ and $\Psi$, including $\nu$ and $\nu_0$.
\end{lemma}

Theorem 4.2 of \cite{narcowich2006sobolev} states that under the conditions of Lemma \ref{th:escape} in addition to
\begin{eqnarray}\label{floor}
\lfloor \beta\rfloor>d/2,
\end{eqnarray}
we have
\begin{eqnarray}\label{2norm}
\|g-\mathcal{I}_{\Phi,X}g\|_{W^\mu_2(\Omega)}\leq C h^{\beta-\mu}_{X,\Omega}\rho^{\tau-\beta}_{X,\Omega}\|g\|_{W^\beta_2(\mathbb{R}^d)},
\end{eqnarray}
for $0\leq \mu\leq \beta$. As commented by a reviewer, condition (\ref{floor}) can be removed by using Theorem 4.1 of \cite{arcangeli2007extension} in the proof of Theorem 4.2 of \cite{narcowich2006sobolev}; also see Theorem 10 of \cite{wynne2020convergence}.
Having (\ref{2norm}), Lemma \ref{th:escape} is an immediate consequence. Specifically, combining (\ref{2norm}) with the real interpolation theory for Sobolev spaces (See, e.g., Theorem 5.8 and Chapter 7 of \cite{adams2003sobolev}), yields (\ref{infnorm}). An alterative proof of Lemma \ref{th:escape}, also suggested by a reviewer, is given in Section \ref{sec:interpolationRHKS}.

Next we make a remark on the notion of smoothness and the settings of smoothness misspecification. For a deterministic function $g$, we say $g$ has smoothness $\nu_0+d/2$ if $g\in W^{\nu_0+d/2}(\Omega)$. The smoothness misspecification in Lemma \ref{th:escape} is stated as: the smoothness associated with the RKHS is higher than the true smoothness of the function when $\nu_0<\nu$.

Now we turn to the role of $\nu_0$ for a stationary Gaussian process $Z(x)$ with spectral density $f_\Psi$ satisfying Condition \ref{C1}. Unlike the usual perception on the smoothness of deterministic functions, here $\nu_0$ should be interpreted as the mean squared differentiability \citep{stein2012interpolation} of the Gaussian process, which is related to the smoothness of the correlation function $\Psi$.

On the other hand, we can also consider the smoothness of sample paths of $Z(x)$, under the usual definition of smoothness for deterministic functions. It turns out that the sample path smoothness is lower than $\nu_0$ with probability one \citep{driscoll1973reproducing,steinwart2019convergence,kanagawa2018gaussian}. In view of this, Theorem \ref{Th:main} implies that the sample paths of Gaussian processes can escape the $d/2$ smoothness misspecification in terms of the $L_\infty$ norm, disregarding the logarithmic factor. In other words, there exist functions with smoothness less than $\nu_0$ that can be approximated at the rate $O(n^{-\nu_0/d}\sqrt{\log n})$, and the set of such functions is large under the probability measure of a certain Gaussian process.

\section{Simulation studies}\label{sec:simulation}
The objective of this section is to verify whether the rate of convergence given by Corollary \ref{coro:rate} is accurate.
We consider the settings in Example \ref{examplerdsample}. We have shown that under a random sampling over the experimental region $\Omega = [0,1]$, the kriging prediction error has the rate $O_{\mathbb{P}}(n^{\nu-2\nu_0}\log^{\nu+1/2} n)$ for $\nu > \nu_0$.
If grid sampling is used, Corollaries \ref{coro:rate} and \ref{coro:quasirate} show that the error has the order of magnitude $n^{-\nu_0}\log^{1/2} n$ for $\nu > \nu_0$.

We denote the expectation of (\ref{krigingerror}) with random sampling and grid sampling by $\mathcal{E}_{\text{rand}}$ and $\mathcal{E}_{\text{grid}}$, respectively. Our idea of assessing the rate of convergence is described as follows. If the error rates are sharp, we have the approximations
\begin{align*}
\log \mathcal{E}_{\text{rand}}&\approx
   (\nu-2\nu_0) \log n + (\nu+\frac{1}{2})\log \log n + \log c_1,\\
\log \mathcal{E}_{\text{grid}}&\approx -\nu_0\log n + \frac{1}{2}\log \log n +\log c_2,
\end{align*}
for random samplings and grid samplings, respectively, where $c_1,c_2$ are constants. Since $\log \log n$ grows much slower than $\log n$, we can regard the $\log \log n$ term as a constant and get the second approximations
\begin{align}
\log \mathcal{E}_{\text{rand}}&\approx
   (2\nu_0-\nu) \log (1/n) + C_1,\label{simu1}\\
\log \mathcal{E}_{\text{grid}}&\approx \nu_0\log (1/n) +C_2.\label{simu2}
\end{align}
To verify the above formulas via numerical simulations, we can regress $\log \mathcal{E}_{\text{rand}}$ and $\log \mathcal{E}_{\text{grid}}$ on $\log (1/n)$ and examine the estimated slopes. If the bounds are sharp, the estimated slopes should be close to the theoretical assertions $2\nu_0-\nu$ and $\nu_0$, respectively.

In our simulation studies, we consider the sample sizes $n=10k$, for $k=2,3,...,15$. For each $k$, we simulate 100 realizations of a Gaussian process. For a specific realization of a Gaussian process, we generate $10k$ independent and uniformly distributed random points as $X$, and use $\sup_{x \in \Omega_1}|Z(x)-\mathcal{I}_{\Phi,X}Z(x)|$ to approximate the uniform error $\sup_{x \in \Omega}|Z(x)-\mathcal{I}_{\Phi,X}Z(x)|$, where $\Omega_1$ is the first $200$ points of the Halton sequence \citep{niederreiter1992random}. We believe that the points are dense enough so that the approximation can be accurate. Then the regression coefficient is estimated using the least squares method. For grid sampling, we adopt a similar approach with the same number of design points $X$. The results are presented in Table \ref{Tab:simuResults}. The first two columns of Table \ref{Tab:simuResults} show the true and imposed smoothness parameters of the Mat\'ern correlation functions. The fourth and the fifth columns show the convergence rates obtained from the simulation studies and the theoretical analysis, respectively. The sixth column shows the relative difference between the fourth and the fifth columns, given by
$|\text{estimated slope-theoretical slope}|/(\text{theoretical slope}).$ The last column gives the $R$-squared values of the linear regression of the simulated data.

\begin{table}[h]
\centering
\begin{tabular}{|c|c|c|c|c|c|c|}
\hline
$\nu_0$ & $\nu$ & Design &  ES & TS & RD & $R^2$\\
\hline
\multirow{2}{*}{1.1} & \multirow{2}{*}{1.3} & RS & 0.9011 & 0.9 & 0.0012 &0.8579\\
\cline{3-7}
& & GS  & 1.0670 & 1.1 & 0.0300 & 0.9992\\
\hline
\multirow{2}{*}{1.1} & \multirow{2}{*}{2.8} & RS & 0.1653 & -0.6 (No convergence) & - & 0.0308\\
\cline{3-7}
& & GS  & 1.0968 & 1.1 & 0.0030 & 0.9995\\
\hline
\multirow{2}{*}{2.1} & \multirow{2}{*}{2.8} & RS & 1.523 & 1.4 & 0.088 & 0.9834\\
\cline{3-7}
& & GS  & 2.0953 & 2.1 & 0.0022 & 0.9992\\
\hline
\multirow{2}{*}{1.5} & \multirow{2}{*}{3.5} & RS & 0.1083 & -0.5 (No convergence) & - & 0.0991\\
\cline{3-7}
& & GS  & 1.4982 & 1.5 & 0.0012 & 0.9989\\
\hline
\end{tabular}
\caption{{\rm Numerical studies on the convergence rates of kriging prediction with oversmoothed correlation functions. The following abbreviations are used: RS=Random sampling, GS= Grid sampling, ES=Estimated slope, TS=Theoretical slope, RD=relative difference. The relative differences are not computed when the corresponding theoretical slopes are negative.}}
\label{Tab:simuResults}
\end{table}

In the setting of Rows 2, 3, 5-7 and 9 of Table \ref{Tab:simuResults}, our theory suggests the prediction consistency, i.e., $h^{\nu_0}_n\rho_n^{\nu-\nu_0}$ tends to zero. It can be seen that the estimated slopes coincide with our theoretical assertions for these cases. Also, the $R$-squared values for these rows are high, which implies a good model fitting of (\ref{simu1})-(\ref{simu2}). When $h^{\nu_0}_n\rho_n^{\nu-\nu_0}$ goes to infinity,
our simulation results suggest a very slow rate of convergence. Specifically, under the random sampling scheme and $(\nu_0,\nu)=(1.1,2.8)$ and $(\nu_0,\nu)=(1.5,3.5)$, the estimated rates of convergence are $0.1653$ and $0.1083$, respectively. Also, the $R$-squared values are very low. These slow rates and poor model fitting imply that the kriging predictor could be inconsistent. Figure \ref{fig:nu0big} shows the scattered plots of the raw data and the regression lines under the four combinations of $(\nu_0,\nu)$ in Table \ref{Tab:simuResults}.

\begin{figure}[h!]
    \centering
    \begin{subfigure}
        \centering
        \includegraphics[height=2.3in]{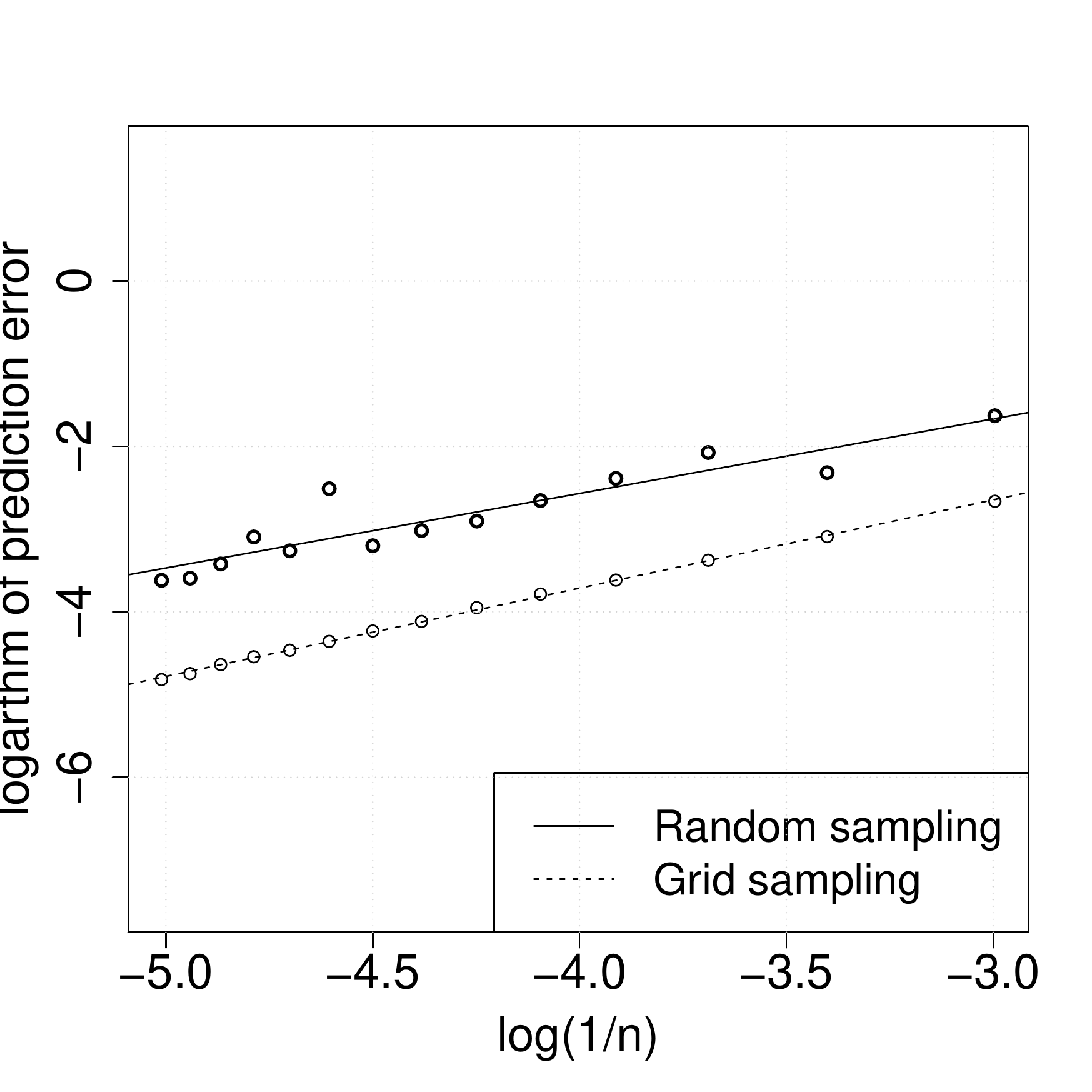}
    \end{subfigure}
    \begin{subfigure}
        \centering
        \includegraphics[height=2.3in]{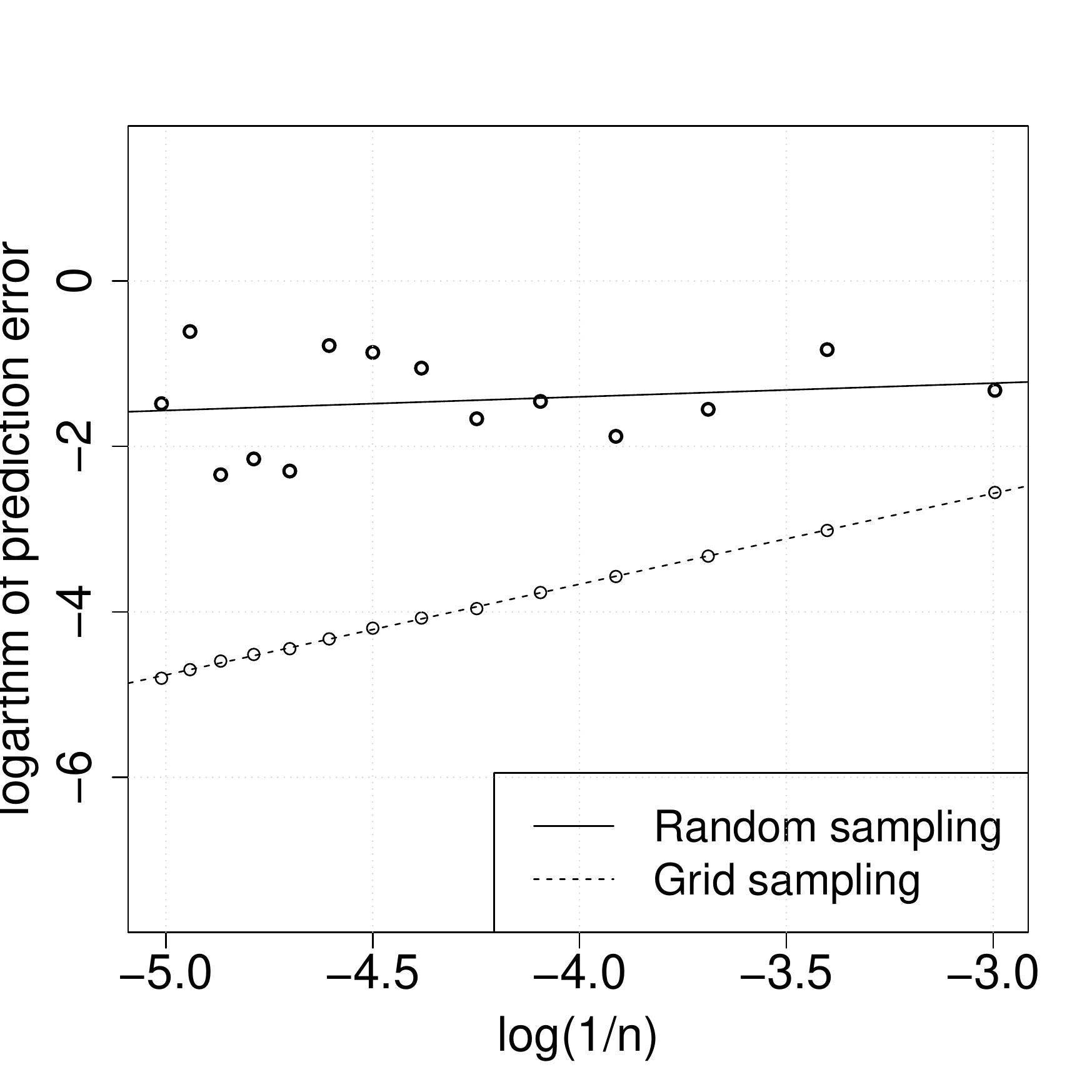}
    \end{subfigure}
    \begin{subfigure}
        \centering
        \includegraphics[height=2.3in]{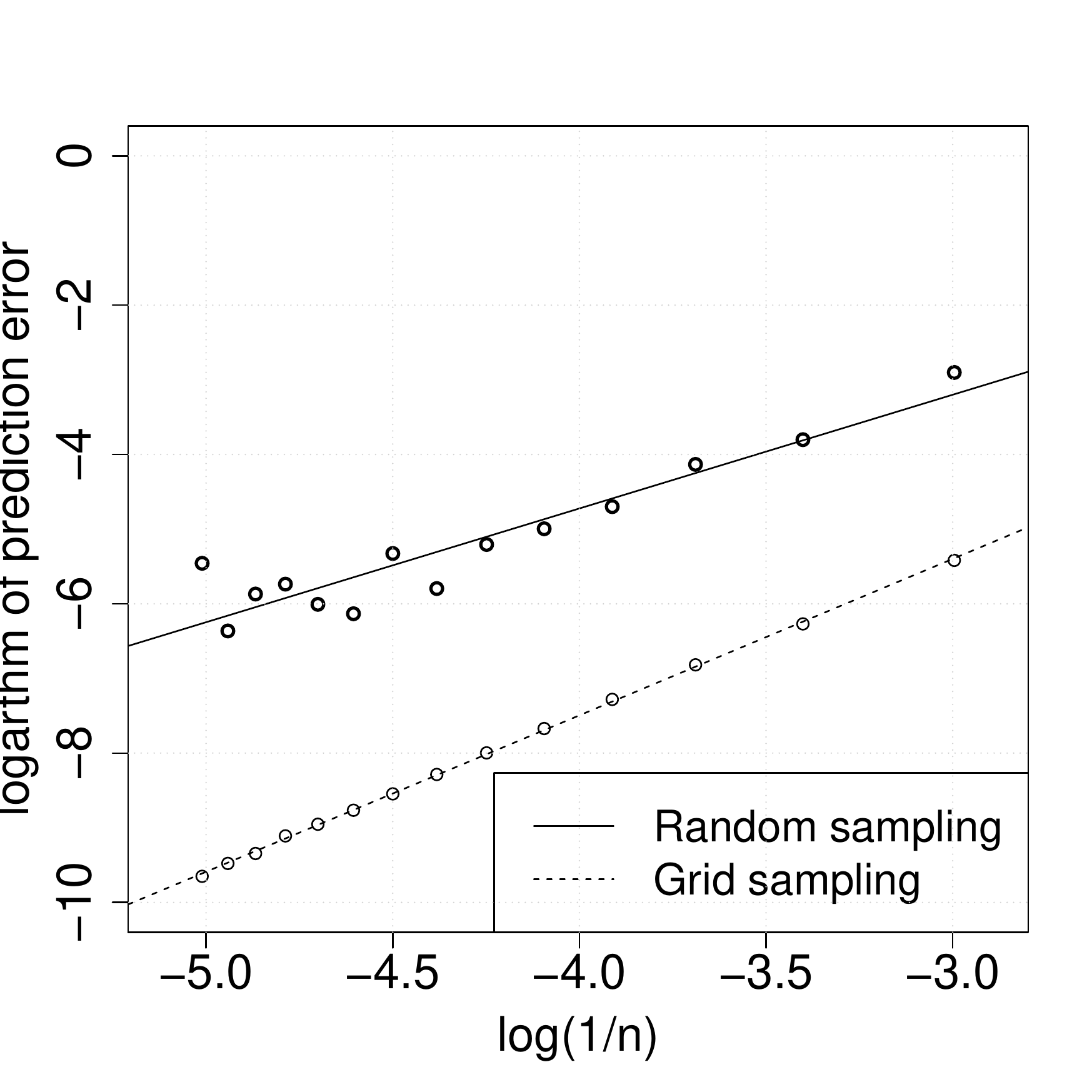}
    \end{subfigure}
    \begin{subfigure}
        \centering
        \includegraphics[height=2.3in]{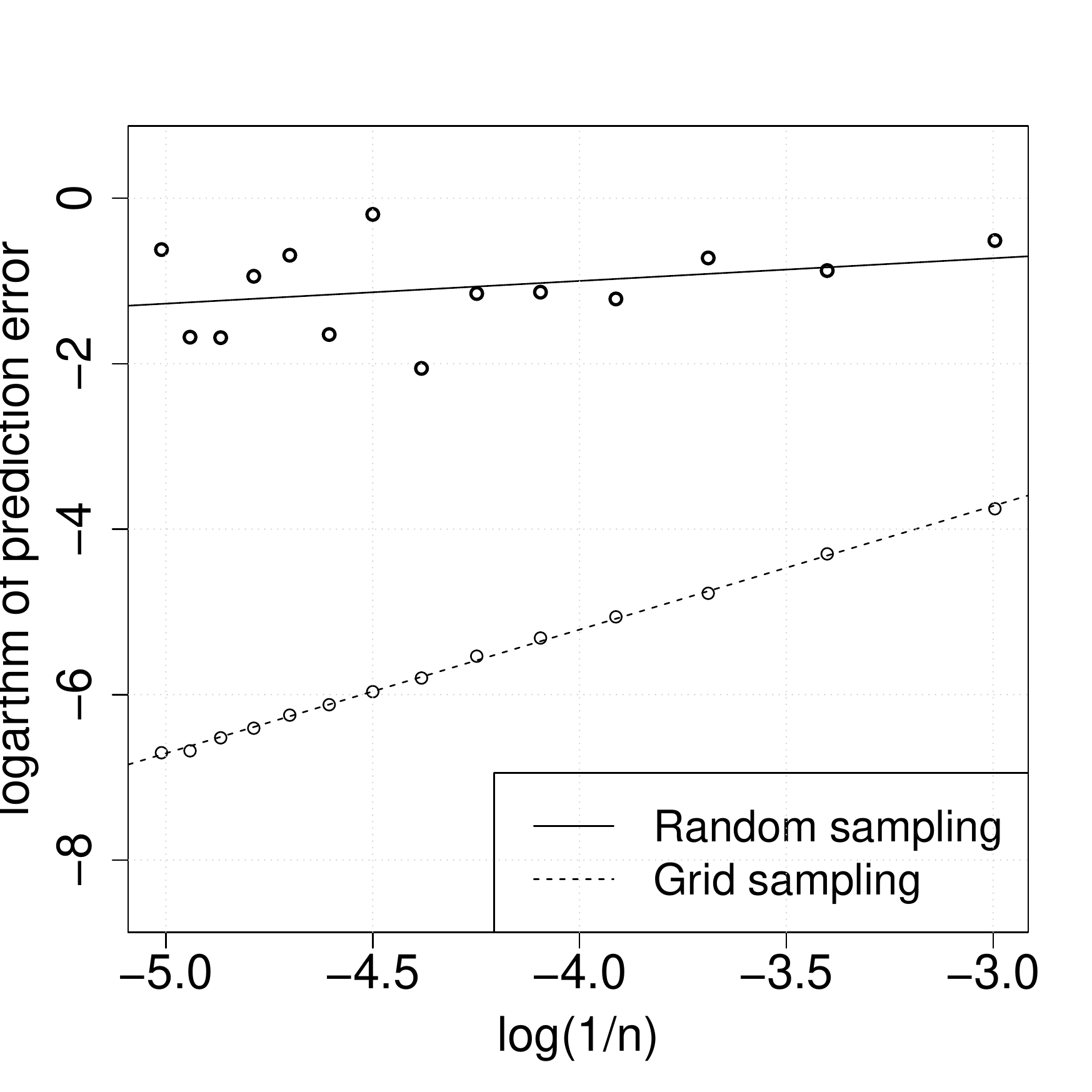}
    \end{subfigure}
   \caption{{\rm The regression line of $\log \mathcal{E}_{\text{unif}}$ and $\log \mathcal{E}_{\text{grid}}$ on $\log(1/n)$, under the four combinations of $(\nu_0,\nu)$ in Table \ref{Tab:simuResults}.
   Each point denotes one average prediction error for each $n$.}}
    \label{fig:nu0big}
\end{figure}

\section{Concluding remarks}

The error bounds presented in this work are not only valuable in mathematics. They can also provide guidelines for practitioners of kriging. Especially, our work confirms the importance of the design of experiments for kriging: if the design is quasi-uniform, the use of an oversmoothed correlation would not be an issue.

It has been known for a while that using quasi-uniform sampling is helpful for deterministic function approximation.
From an approximation theory perspective, one of the main contributions of this work is the discovery that sample paths of Gaussian processes escapes the $d/2$ smoothness misspecification (in the scattered data approximation sense \citep{kanagawa2018gaussian}).

As a final remark, we compare the rates in this work with the ones in radial basis function approximation \citep{edmunds2008function,wendland2004scattered}. For the radial basis function approximation problems, we adopt the standard framework so that the underlying function lies in the reproducing kernel Hilbert space generated by the correlation function. For the $L_\infty$ norm, the obtained optimal rate of convergence for kriging is $O_\mathbb{P}(n^{-\nu_0/d}\sqrt{\log n})$; while that for the radial basis function approximation is $O(n^{-\nu_0/d})$. So there is a difference in the $\sqrt{\log n}$ factor. For $L_p$ norms with $1\leq p<\infty$, the difference is more dramatic. While the optimal rate of convergence for kriging is $O_\mathbb{P}(n^{-\nu_0/d})$, that for radial basis function approximation is $O(n^{-\nu_0/d-\min(1/2,1/p)})$. This gap between the optimal rates can be explained, as the support of a Gaussian process is essentially larger than the corresponding reproducing kernel Hilbert space \citep{van2008reproducing}.

\section{Proofs}\label{sec:proof}

This section comprises our technical proofs. The proofs rely on some results in scattered data approximation of functions in reproducing kernel Hilbert spaces. We introduce these results in Section \ref{sec:RKHS}. The proofs of the theorems in Sections \ref{sec:mainLinfty} and \ref{sec:bounds} are given in Sections \ref{sec:proofLinfty} and \ref{sec:proofLp}, respectively.

Before introducing the details, we first note that in the proofs of all results in Sections \ref{sec:mainLinfty} and \ref{sec:bounds}, it suffices to consider only the case with $\sigma^2=1$. This should not affect the general result because otherwise we can consider the Gaussian process $Z/\sigma$ instead of $Z$.
Thus for notational simplicity, we assume $\sigma^2=1$ throughout this section.

\subsection{Reproducing kernel Hilbert spaces and scattered data approximation}\label{sec:RKHS}

We adopt one reviewer's suggestions to prove our main results using techniques from reproducing kernel Hilbert spaces and recent developments in scattered data approximation, in lieu of our original technique of Fourier transform calculations in the previous version. The current treatment can streamline the proofs, and better show how the intermediate quantities toward the error analysis for Gaussian process regression are linked to those studied in scattered data approximation. Reproducing kernel Hilbert spaces is a common mathematical tool in Gaussian processes and scattered data approximation.

\begin{defn}
    Given a positive definite kernel $K(\cdot)$, the reproducing kernel Hilbert space (RKHS) $\mathcal{N}_K(\mathbb{R}^d)$ is defined as the completion of the function space
$$\left\{\sum_{j=1}^N\beta_j K(\cdot-x_j):N\in\mathbb{N},\beta_j\in\mathbb{R},x_j\in\mathbb{R}^d\right\} $$
under the inner product
\begin{align}\label{RKHSinnerproduct}
 \left\langle\sum_{j=1}^N\beta_j K(\cdot-x_j),\sum_{k=1}^{N'}\beta'_k K(\cdot-x'_k)\right\rangle_K=\sum_{j=1}^N\sum_{k=1}^{N'}\beta_j\beta'_k K(x_j-x'_k).
\end{align}
Denote the RKHS norm by $\|\cdot\|_K$.
\end{defn}

\subsubsection{Interpolation in RKHSs}\label{sec:interpolationRHKS}

We first consider the interpolation of a function $f\in\mathcal{N}_\Psi(\mathbb{R}^d)$, by $\mathcal{I}_{\Psi,X}f$. We have the following known results. Lemmas \ref{lemRKHSnorms} and \ref{lemEquivalence} are Corollaries 10.25 and 10.48 of \cite{wendland2004scattered}, respectively.

\begin{lemma}\label{lemRKHSnorms}
For any $f\in\mathcal{N}_\Psi(\mathbb{R}^d)$,
$\|f-\mathcal{I}_{\Psi,X}f\|_{\Psi}\leq \|f\|_\Psi.$
\end{lemma}

\begin{lemma}\label{lemEquivalence}
Under Condition \ref{C1}, $\mathcal{N}_\Psi(\mathbb{R}^d)=W^{\nu_0+d/2}_2(\mathbb{R}^d)$ with equivalent norms.
\end{lemma}

A reviewer suggested an alternative proof of Lemma \ref{th:escape}, by leveraging the following Lemma \ref{lem:bandlimited} from \cite{narcowich2006sobolev}. It is worth presenting the proof here, because we will later employ Lemma \ref{lem:bandlimited} again.

\begin{lemma}[Theorem 3.4 of \cite{narcowich2006sobolev}] \label{lem:bandlimited}
Suppose $\nu\geq \nu_0>0$. Then for each $g\in W_2^{\nu_0+d/2}(\mathbb{R}^d)$, there exists $g_\gamma\in W_2^{\nu+d/2}(\mathbb{R}^d)$, so that $g|_X=g_\gamma|_X$ and
\begin{align*}
\|g_\gamma\|_{W^{\nu+d/2}_2(\mathbb{R}^d)}&\leq C q^{-(\nu-\nu_0)}_X \|g\|_{{W^{\nu_0+d/2}_2(\mathbb{R}^d)}},\\ \|g_\gamma\|_{W^{\nu_0+d/2}_2(\mathbb{R}^d)}&\leq C \|g\|_{W^{\nu_0+d/2}_2(\mathbb{R}^d)},
\end{align*}
for a constant $C$ depending only on $d$ and $\nu_0$.
\end{lemma}

\begin{proof}[Proof of Lemma \ref{th:escape}]
For each $g\in W_2^{\nu_0+d/2}(\mathbb{R}^d)$, let $g_\gamma$ be the function given in Lemma \ref{lem:bandlimited}. The condition $g|_X=g_\gamma|_X$ implies that $\mathcal{I}_{\Phi,X}g=\mathcal{I}_{\Phi,X}g_\gamma$ and $\mathcal{I}_{\Psi,X}g=\mathcal{I}_{\Psi,X}g_\gamma$.

Corollary 11.33 of \cite{wendland2004scattered} asserts that
\begin{align}\label{samplingineq}
\|f-\mathcal{I}_{\Phi,X}f\|_{L_\infty(\mathbb{R}^d)}\leq C h_{X,\Omega}^{\nu}\|f\|_{W^{\nu+d/2}_2(\mathbb{R}^d)}. \end{align}
Now by triangle inequality,
\begin{align*}
    \|g-\mathcal{I}_{\Phi,X}g\|_{L_\infty(\mathbb{R}^d)} \leq & \|g-\mathcal{I}_{\Psi,X}g\|_{L_\infty(\mathbb{R}^d)} + \|g_\gamma-\mathcal{I}_{\Psi,X}g_\gamma\|_{L_\infty(\mathbb{R}^d)} + \|g_\gamma-\mathcal{I}_{\Phi,X}g_\gamma\|_{L_\infty(\mathbb{R}^d)}\\
    \leq & C_1h_{X,\Omega}^{\nu_0}\left(\|g\|_{W_2^{\nu_0+d/2}(\mathbb{R}^d)} + \|g_\gamma\|_{W_2^{\nu_0+d/2}(\mathbb{R}^d)}\right) + C_2 h_{X,\Omega}^\nu \|g_\gamma\|_{W_2^{\nu+d/2}(\mathbb{R}^d)}\\
    \leq & C_3(h_{X,\Omega}^{\nu_0} + \rho_{X,\Omega}^{\nu-\nu_0}h_{X,\Omega}^{\nu_0})\|g\|_{W_2^{\nu_0+d/2}(\mathbb{R}^d)}\\
    \leq & C_4h_{X,\Omega}^{\nu_0}\rho_{X,\Omega}^{\nu-\nu_0}\|g\|_{W_2^{\nu_0+d/2}(\mathbb{R}^d)},
\end{align*}
where the second inequality follows from (\ref{samplingineq}) and an equivalent form of (\ref{samplingineq}) by replacing $\Phi$ with $\Psi$ and $\nu$ with $\nu_0$.
Hence the proof is completed.
\end{proof}

\subsubsection{Quasi-power functions}

Lemma \ref{lemGPRKHS} states a simple connection between Gaussian processes and RKHSs.

\begin{lemma}\label{lemGPRKHS}
Let $G(\cdot)$ be a stationary Gaussian process on $\Omega$ with a unit variance and a positive definite correlation function $K$. Then for $x_1,\ldots, x_N\in\Omega$ and $\beta_1,\ldots,\beta_N\in\mathbb{R}$,
\begin{align}
Var\left(\sum_{j=1}^N \beta_j G(x_j)\right)&=\left\|\sum_{j=1}^N \beta_j K(\cdot-x_j)\right\|^2_K\label{lemma:GPRKHS1}\\
&=\sup_{\|f\|_K\leq 1}\left|\sum_{j=1}^N\beta_j f(x_j)\right|^2\label{lemma:GPRKHS2}.
\end{align}
\end{lemma}

\begin{proof}
Equation (\ref{lemma:GPRKHS1}) follows from direct calculations using (\ref{RKHSinnerproduct}); equation (\ref{lemma:GPRKHS2}) is Lemma 3.9 of \cite{kanagawa2018gaussian}.
\end{proof}

Recall that a kriging interpolant is defined as $\mathcal{I}_{\Phi,X}Z(x)=r^T_\Phi(x)K_{\Phi}^{-1}Y$; see (\ref{interpolant}).
Lemma \ref{lemGPRKHS} will be employed by partially choosing $\beta_j$'s as the coefficients of a kriging interpolant, i.e., $(\beta_1,\ldots,\beta_n)=r^T_\Phi(x)K^{-1}_\Phi$, which is indeed a constant vector given $x$ and $X$. For example, Lemma \ref{lemGPRKHS} implies
\begin{align}
    \mathbb{E}[Z(x)-\mathcal{I}_{\Phi,X}Z(x)]^2&=\|\Psi(\cdot-x)-\mathcal{I}_{\Phi,X}\Psi(\cdot-x)\|_\Psi^2\nonumber\\
    &=\sup_{\|f\|_\Psi\leq 1}|f(x)-\mathcal{I}_{\Phi,X}f(x)|^2\label{quasipower}.
\end{align}
We shall call the quantity in (\ref{quasipower}) the \textit{quasi-power function}, denoted as $Q^2(x)$. Note that $Q^2(x)$ should also depend on $\Phi,\Psi$ and $X$, but we suppress this dependence for notational simplicity, and this will cause no ambiguity. A related quantity is the \textit{power function} \citep{wendland2004scattered}, defined as $$P^2_{\Psi,X}(x):=\mathbb{E}[Z(x)-\mathcal{I}_{\Psi,X}Z(x)]^2,$$ which is the conditional variance of $Z(x)$ given $Z(x_1),\ldots,Z(x_n)$. A simple relationship between $Q(x)$ and $P_{\Psi,X}(x)$ is
\begin{align}\label{BLP}
    Q(x)\geq P_{\Psi,X}(x).
\end{align}
This inequality can be proven via elementary calculations by showing that $\mathcal{I}_{\Psi,X}Z$ has the smallest mean squared prediction error among all predictors in terms of linear combinations of $Z(x_j)$, and $\mathcal{I}_{\Phi,X}Z$ is one of such. This result is also known as the \textit{best linear prediction} property of $\mathcal{I}_{\Psi,X}Z$ \citep{stein2012interpolation,santner2013design}.

The interest here lies in bounding $Q(x)$ in different manners. We state the results in Sections \ref{sec:Qupper} and \ref{sec:Qlower}.

\subsubsection{Upper bounds of the quasi-power function}\label{sec:Qupper}

Lemma \ref{th:Q} can be proven immediately by putting together Lemmas \ref{th:escape}, \ref{lemEquivalence} and \ref{lemGPRKHS}. Lemma \ref{th:Qunder} is a counterpart of Lemma \ref{th:Q} under the condition $\nu\leq \nu_0$, which follows directly from Lemmas \ref{th:escape}, \ref{lemEquivalence} and (\ref{samplingineq}).

\begin{lemma}\label{th:Q}
	Suppose Conditions \ref{C4}-\ref{C3} are met. If $\nu\geq\nu_0$, then there exist constants $C>0$ and $h_0\in(0,1]$ independent of $X$ and $x$ such that
	$$Q(x)\leq C h^{\nu_0}_{X,\Omega}\rho^{\nu-\nu_0}_{X,\Omega} $$
	holds for all $x\in\Omega$ and all $X$ satisfying $h_{X,\Omega}\leq h_0$.
\end{lemma}

\begin{lemma}\label{th:Qunder}
Suppose Conditions \ref{C4}-\ref{C3} are met. If $\nu\leq\nu_0$, then there exist constants $C>0$ and $h_0\in(0,1]$ independent of $X$ and $x$ such that
	$$Q(x)\leq C h^{\nu}_{X,\Omega}$$
	holds for all $x\in\Omega$ and all $X$ satisfying $h_{X,\Omega}\leq h_0$.
\end{lemma}

\subsubsection{A lower bound of the quasi-power function}\label{sec:Qlower}

The goal of this section is to prove a lower bound of the quasi-power function under the $L_2(\Omega)$ norm, given by Lemma \ref{lem:Qlower}.

\begin{lemma}\label{lem:Qlower}
Suppose Conditions \ref{C4}-\ref{C1} are met. Then we have
$$\|P_{\Psi,X}\|_{L_2(\Omega)}\geq C n^{-\nu_0/d}, $$
where $n={\rm card}(X)$.
Here the constant $C$ depends only on $\Omega$ and $\Psi$, including $\nu_0$.
\end{lemma}

Because $\|P_{\Psi,X}\|_{L_2(\Omega)}\leq \sqrt{{\rm Vol}(\Omega)}\sup_{x\in\Omega}P_{\Psi,X}(x)$, where ${\rm Vol}(\Omega)$ denotes the volume of $\Omega$, we obtain Corollary \ref{coro:Qlower}. Corollary \ref{coro:Qlower} is a standard result in scattered data approximation; see, for example, Theorem 11 of \cite{wenzel2019novel}.

\begin{corollary}\label{coro:Qlower}
Suppose Conditions \ref{C4}-\ref{C1} are met. Then we have
$$\sup_{x\in\Omega}P_{\Psi,X}(x)\geq C n^{-\nu_0/d}. $$
Here the constant $C$ depends only on $\Omega$ and $\Psi$, including $\nu_0$.
\end{corollary}

To prove Lemma \ref{lem:Qlower}, we need a result from the average-case analysis of numerical problems, given by
Lemma \ref{th:lowLp}, which is a direct consequence of Theorem 1.2 of \cite{papageorgiou1990average}. It states lower bounds of $\|Q\|_{L_2(\Omega)}$ in terms of the eigenvalues.

Because $\Psi$ is a positive definite function, by Mercer's theorem (see \cite{Pog1966integral} for example), there exists a countable set of positive eigenvalues $\lambda_1\geq \lambda_2\geq...>0$ and an orthonormal basis for $L_2(\Omega)$, denoted as $\{\varphi_k\}_{k\in\mathbb{N}}$, such that
\begin{align}\label{eigenmercer}
\Psi(x-y) = \sum_{k=1}^\infty \lambda_k \varphi_k(x)\varphi_k(y),
\end{align}
where the summation is uniformly and absolutely convergent.

\begin{lemma}\label{th:lowLp}
Let $\lambda_k$'s be eigenvalues of $\Psi$. Then we have
\begin{eqnarray*}
\|P_{\Psi,X}\|^2_{L_2(\Omega)} \geq
     \sum_{k=n+1}^\infty \lambda_k.
\end{eqnarray*}
\end{lemma}

\begin{proof}[Proof of Lemma \ref{lem:Qlower}]
Define the $k$th approximation number of the embedding $id: W_2^{\nu_0+\frac{d}{2}}(\Omega) \rightarrow L_{p}(\Omega)$, denoted by $a_k$, by
\begin{align*}
    a_k = \inf \{\|id - L\|, H \in \mathcal{H}(W_2^{\nu_0+\frac{d}{2}}(\Omega), L_{2}(\Omega)), \text{rank}(H)<k\},
\end{align*}
where $\mathcal{H}(W_2^{\nu_0+\frac{d}{2}}(\Omega), L_{2}(\Omega))$ is the family of all bounded linear mappings $W_2^{\nu_0+\frac{d}{2}}(\Omega) \rightarrow L_{2}(\Omega)$, $\|\cdot\|$ is the operator norm, and $\text{rank}(H)$ is the dimension of the range of $H$ \citep{edmunds2008function}. The approximation number measures the approximation properties by affine (linear) $k$-dimensional mappings. Let $T$ be the embedding operator of $\mathcal{N}_{\Psi}(\Omega)$ into $L_2(\Omega)$, and $T^*$ be the adjoint of $T$. By Proposition 10.28 in \cite{wendland2004scattered},
\begin{align*}
T^*v(x) = \int_{ \Omega} \Psi(x-y)v(y)dy,\quad v\in L_2( \Omega), x\in \Omega.
\end{align*}
By Lemma \ref{lemEquivalence}, $W_2^{\nu_0+\frac{d}{2}}( \Omega)$ coincides with $\mathcal{N}_{\Psi}( \Omega)$. By Theorem 5.7 in \cite{edmunds1987spectral}, $T$ and $T^*$ have the same singular values. By Theorem 5.10 in \cite{edmunds1987spectral}, for all $k\in \mathbb{N}$, $a_k(T)=\mu_k(T)$, where $a_k(T)$ denotes the approximation number for the embedding operator (as well as the integral operator), and $\mu_k$ denotes the singular value of $T$. By the theorem in Section 3.3.4 in \cite{edmunds2008function}, the embedding operator $T$ has approximation numbers
satisfying
\begin{align}\label{ineqnAppNum}
C_3k^{-\nu_0/d-1/2}\leqslant a_k\leqslant C_4k^{-\nu_0/d-1/2}, \forall k\in \mathbb{N},
\end{align}
where $C_3$ and $C_4$ are two positive numbers. By Theorem 5.7 in \cite{edmunds1987spectral}, $T^*T\varphi_k=\mu_k^2\varphi_k$, and $T^*T\varphi_k=T^*\varphi_k=\lambda_k\varphi_k$, we have $\lambda_k=\mu_k^2$. By (\ref{ineqnAppNum}), $\lambda_k\asymp k^{-2\nu_0/d-1}$ holds. Then the desired result follows from Lemma \ref{th:lowLp}.
\end{proof}

\subsection{$L_\infty$ results}\label{sec:proofLinfty}
In this section, we prove Theorems \ref{Th:main}  and \ref{th:low}.
The natural distances of Gaussian processes play a crucial role in establishing these $L_\infty$ results.

\begin{defn}
    The natural distance $d(x,x')$ of a zero-mean Gaussian process $G(x)$ with $x\in\Omega$ is defined as
    $$d^2_G(x,x')=\mathbb{E}[G(x)-G(x')]^2, $$
    for $x,x'\in \Omega$. Once equipped with $d_G$, $\Omega$ becomes a metric space.
\end{defn}

The \textit{$\epsilon$-covering number} of the metric space $(\Omega,d_G)$, denoted as $N(\epsilon,\Omega,d_G)$, is the minimum integer $N$ so that there exist $N$ distinct balls in $(\Omega,d_G)$ with radius $\epsilon$, and the union of these balls covers $\Omega$. The natural distance and the associated covering number, are closely tied to the $L_\infty$ norm of the Gaussian process, say $\sup_{x\in\Omega}|G(x)|$. The needed results are collected in Lemmas \ref{concentrationGP}-\ref{lbexgau}. Lemma \ref{concentrationGP} is a version of the Borell-TIS inequality for the $L_\infty$ norm of a Gaussian process. Its proof can be found in \cite{pisier1999volume}.

\begin{lemma}[Borell-TIS inequality]\label{concentrationGP}
Let $G(x)$ be a separable zero-mean Gaussian process with continuous sample paths almost surely and $x$ lying in a $d_G$-compact set $\Omega$. Let $\sigma_G^2 = \sup_{x\in \Omega}\mathbb{E}G(x)^2$. Then, we have $\mathbb{E}\sup_{x\in \Omega} |G(x)| < \infty$ and for all $t>0$,
\begin{align}
&\mathbb{P}\bigg( \mathbb{E} \sup_{x\in \Omega} |G(x)| -  \sup_{x\in \Omega} |G(x)| \geq t\bigg)\leq e^{-t^2/2\sigma_G^2},\label{concentrationGP1}\\
&\mathbb{P}\bigg( \mathbb{E} \sup_{x\in \Omega} |G(x)| -  \sup_{x\in \Omega} |G(x)| \leq -t\bigg)\leq e^{-t^2/2\sigma_G^2}.\label{concentrationGP2}
\end{align}
\end{lemma}

\begin{lemma}[Corollary 2.2.8 of \cite{van1996weak}]\label{lem:GPsupnorm}
Let $G(x)$ be as in Lemma \ref{concentrationGP}. For some universal constant $C$, we have
\begin{align*}
\mathbb{E} \sup_{ x,x'\in \Omega} |G( x)-G(x')| \leq C \int_0^D \sqrt{\log N(\epsilon,\Omega,d_G)}d\epsilon,
\end{align*}
where $D=\sup_{x,x'\in\Omega}d_G(x,x')$ is the diameter of $\Omega$ under $d_G$.
\end{lemma}

\begin{lemma}[Theorem 6.5 of \cite{van2014probability}]\label{lbexgau}
Let $G(x)$ be as in Lemma \ref{concentrationGP}. For some universal constant $C$, we have
\begin{align*}
\mathbb{E} \sup_{ x\in \Omega} |G( x)| \geq C \sup_{\eta>0} \eta \sqrt{\log N(\eta,\Omega,d_G)}.
\end{align*}
\end{lemma}

To utilize the above lemmas to bound $\sup_{x\in\Omega}|Z(x)-\mathcal{I}_{\Phi,X}Z(x)|$, the main idea is to note that $$g_Z(x):=Z(x)-\mathcal{I}_{\Phi,X}Z(x)$$ is also a Gaussian process. So Lemma \ref{concentrationGP} can be applied directly. The remainder is to bound $\mathbb{E}\sup_{x\in\Omega}|Z(x)-\mathcal{I}_{\Phi,X}Z(x)|$. According to Lemmas \ref{lem:GPsupnorm} and \ref{lbexgau}, it is crucial to understand the natural distance, given by
$$d^2_{g_Z}(x,x')=\mathbb{E}[Z(x)-\mathcal{I}_{\Phi,X}Z(x)-Z(x')+\mathcal{I}_{\Phi,X}Z(x')]^2.$$

\subsubsection{Proof of Theorem \ref{Th:main}}
The main steps of proving Theorem \ref{Th:main} are: 1) bounding the diameter $D$; 2) connecting the natural distance $d_{g_Z}$ with the Euclidean distance; 3) bounding the covering integral and establishing the desired result.

\textbf{Step 1.} An upper bound of $D$ is given by
\begin{align}
    D^2&=\sup_{x,x'\in\Omega}\mathbb{E}[Z(x)-\mathcal{I}_{\Phi,X}Z(x)-Z(x')+\mathcal{I}_{\Phi,X}Z(x')]^2\nonumber\\
    &\leq 4\sup_{x\in\Omega}\mathbb{E}[Z(x)-\mathcal{I}_{\Phi,X}Z(x)]^2\nonumber\\
    &=4\sup_{x\in\Omega}Q^2(x)\leq C_1^2 h^{2\nu_0}_{X,\Omega}\rho^{2(\nu-\nu_0)}_{X,\Omega},\label{UB_D}
\end{align}
where the first inequality follows from the basic inequality $(x+y)^2\leq 2x^2+2y^2$; the last inequality follows from Lemma \ref{th:Q}.

\textbf{Step 2.} By Lemma \ref{lemGPRKHS},
\begin{align}\label{ND1}
    d_{g_Z}(x,x')=\sup_{\|f\|_\Psi\leq 1}|f(x)-\mathcal{I}_{\Phi,X}f(x)-f(x')+\mathcal{I}_{\Phi,X}f(x')|.
\end{align}
The H\"older space $C^{0,\alpha}_b(\mathbb{R}^d)$ for $0< \alpha\leq 1$ consists of continuous bounded function on $\mathbb{R}^d$, with its norm defined as
$$\|f\|_{C_b^{0,\alpha}(\mathbb{R}^d)}:=\sup_{x,x'\in\mathbb{R}^d,x\neq x'}\frac{|f(x)-f(x')|}{\|x-x'\|^\alpha}. $$
For $f\in\mathcal{N}_\Psi(\mathbb{R}^d)$,
Lemma \ref{lemEquivalence} implies that $f-\mathcal{I}_{\Phi,X}f\in W_2^{\nu_0+d/2}(\mathbb{R}^d)$. The Sobolev embedding theorem (see, for example, Theorem 4.47 of \cite{demengel2012functional}) implies the embedding relationship $ W^{\nu_0+d/2}_2(\mathbb{R}^d)\subset C_b^{0,\tau}(\mathbb{R}^d)$ with $\tau=\min(\nu_0,1)$, and
\begin{align}\label{embedding}
\|h\|_{C_b^{0,\tau}(\mathbb{R}^d)}\leq C_2\|h\|_{W^{\nu_0+d/2}_2(\mathbb{R}^d)},
\end{align}
for all $h\in W^{\nu_0+d/2}_2(\mathbb{R}^d)$ and a constant $C_2$. Therefore, we have $f-\mathcal{I}_{\Phi,X}f\in C_b^{0,\tau}(\mathbb{R}^d)$.

Thus by (\ref{ND1}), we have
\begin{align}
    d_{g_Z}(x,x')&\leq \sup_{\|f\|_\Psi\leq 1}\|f-\mathcal{I}_{\Phi,X}f\|_{C_b^{0,\tau}(\mathbb{R}^d)}\|x-x'\|^\tau\nonumber\\
    &\leq \sup_{\|f\|_\Psi\leq 1}C_2\|f-\mathcal{I}_{\Phi,X}f\|_{W^{\nu_0+d/2}_2(\mathbb{R}^d)}\|x-x'\|^\tau\nonumber\\
    &\leq \sup_{\|f\|_\Psi\leq 1}C_3\|f-\mathcal{I}_{\Phi,X}f\|_{\Psi}\|x-x'\|^\tau,\label{dbound1}
\end{align}
where the second inequality follows from (\ref{embedding}); the third inequality follows from Lemma \ref{lemEquivalence}.

Now we employ Lemma \ref{lem:bandlimited} again. Let $f_\gamma$ be the function asserted by Lemma \ref{lem:bandlimited} with $f|_X=f_\gamma|_X$. Similar to the proof of Lemma \ref{th:escape}, we have
\begin{align}
    &\|f-\mathcal{I}_{\Phi,X}f\|_{\Psi}\nonumber\\
    \leq &\|f-\mathcal{I}_{\Psi,X}f\|_{\Psi}+\|f_\gamma-\mathcal{I}_{\Psi,X}f_\gamma\|_{\Psi}+\|f_\gamma-\mathcal{I}_{\Phi,X}f_\gamma\|_{\Psi}\nonumber\\
    \leq &\|f\|_\Psi+\|f_\gamma\|_\Psi+\|f_\gamma-\mathcal{I}_{\Phi,X}f_\gamma\|_{\Psi},\nonumber\\
    \leq & C_4\|f\|_\Psi+\|f_\gamma-\mathcal{I}_{\Phi,X}f_\gamma\|_{\Psi}\label{normbound1}
\end{align}
where the second inequality follows from Lemma \ref{lemRKHSnorms}; the last inequality follows from Lemmas \ref{lemEquivalence} and \ref{lem:bandlimited}. Similarly, we have
\begin{align*}
&\|f_\gamma-\mathcal{I}_{\Phi,X}f_\gamma\|_{\Psi}\leq C_5\|f_\gamma-\mathcal{I}_{\Phi,X}f_\gamma\|_{W^{\nu_0+d/2}_2(\mathbb{R}^d)}\\
\leq& C_5 \|f_\gamma-\mathcal{I}_{\Phi,X}f_\gamma\|_{W^{\nu+d/2}_2(\mathbb{R}^d)}  \leq   C_6 \|f_\gamma-\mathcal{I}_{\Phi,X}f_\gamma\|_{\Phi}\\
\leq & C_6\|f_\gamma\|_\Phi \leq C_7 \|f_\gamma\|_{W^{\nu+d/2}_2(\mathbb{R}^d)}\leq  C_8 q_X^{-(\nu-\nu_0)}\|f\|_{W^{\nu_0+d/2}_2(\mathbb{R}^d)}\\
\leq & C_9 q_X^{-(\nu-\nu_0)}\|f\|_\Psi,
\end{align*}
which, together with (\ref{dbound1}) and (\ref{normbound1}), yields
$$d_{g_Z}(x,x')\leq C_{10} q^{-(\nu-\nu_0)}_X\|x-x'\|^\tau=C_{10}h^{\nu_0-\nu}_{X,\Omega}\rho_{X,\Omega}^{\nu-\nu_0}\|x-x'\|^\tau. $$

Therefore, by the definition of the covering number, we have
\begin{eqnarray}\label{e1new}
N(\epsilon,\Omega,d_{g_Z})\leq N((\epsilon/ C_{10}h^{\nu_0-\nu}_{X,\Omega}\rho_{X,\Omega}^{\nu-\nu_0})^{1/\tau},\Omega,\|\cdot\|).
\end{eqnarray}
The right side of (\ref{e1new}) involves the covering number of a Euclidean ball, which is studied in the literature; see Lemma 2.5 of \cite{geer2000empirical}. This result leads to the bound
\begin{align}\label{eq:entropyBoundC1}
\log N(\epsilon,\Omega,d_{g_Z})\leq C_{11}\log\bigg( 1 + C_{12}\left(\frac{ h^{\nu_0-\nu}_{X,\Omega}\rho_{X,\Omega}^{\nu-\nu_0}}{\epsilon}\right)^{1/\tau}\bigg).
\end{align}

\textbf{Step 3.} For any $x_1\in X$, the interpolation property implies $g_Z(x_1)=0$. Using our findings in Steps 1 and 2, together with Lemma \ref{lem:GPsupnorm}, we have
\begin{align}\label{ieqintpart1}
    &\mathbb{E}\sup_{x\in\Omega}|g_Z(x)|=\mathbb{E}\sup_{x\in\Omega}|g_Z(x)-g_Z(x_1)|\nonumber\\
    \leq &\mathbb{E}\sup_{x,x'\in\Omega}|g_Z(x)-g_Z(x')|\nonumber\\
    \leq & C_{13} \int_0^{C_1 h_{X,\Omega}^{\nu_0}\rho^{\nu-\nu_0}_{X,\Omega}}\sqrt{\log\bigg( 1 + C_{12}\left(\frac{ h^{\nu_0-\nu}_{X,\Omega}\rho_{X,\Omega}^{\nu-\nu_0}}{\epsilon}\right)^{1/\tau}\bigg)} d\epsilon\nonumber\\
    = & C_{13}h_{X,\Omega}^{\nu_0}\rho^{\nu-\nu_0}_{X,\Omega} \int_0^{C_1} \sqrt{\log\bigg( 1 + C_{12}\left(\frac{ h^{-\nu}_{X,\Omega}}{t}\right)^{1/\tau}\bigg)} dt,
\end{align}
where the second equality is obtained by the change of variables. Note that for any $b>\frac{1}{C_1}$ and $a>0$, taking $C'=\max\{C_1,1\}$ leads to $1 + b^a \leq (1+C'b)^a + C'b(1+C'b)^a\leq (1+C'b)^{a+1}$. Thus we have
\begin{align*}
\log\bigg( 1 + C_{12}\left(\frac{ h^{-\nu}_{X,\Omega}}{t}\right)^{1/\tau}\bigg)
    \leq &\bigg(1 + \frac{1}{\tau}\bigg)\log\bigg( 1 + C_{14}\frac{ h^{-\nu}_{X,\Omega}}{t}\bigg)
\end{align*}
for $t\in (0,C_1]$.

Therefore, the integral \eqref{ieqintpart1} can be further bounded by
\begin{align}\label{ieqintpart2}
     C_{15}h_{X,\Omega}^{\nu_0}\rho^{\nu-\nu_0}_{X,\Omega} \int_0^{C_1} \sqrt{\log\bigg( 1 + \frac{C_{14}}{h^{\nu}_{X,\Omega}t}\bigg)} dt.
\end{align}
We then apply the Cauchy-Schwarz inequality to get
\begin{align}\label{ieqintpart3}
     & C_{15}h_{X,\Omega}^{\nu_0}\rho^{\nu-\nu_0}_{X,\Omega} \int_0^{C_1} \sqrt{\log\bigg( 1 + \frac{C_{14}}{h^{\nu}_{X,\Omega}t}\bigg)} dt\nonumber\\
     \leq & C_{15}C_1^{1/2}h_{X,\Omega}^{\nu_0}\rho^{\nu-\nu_0}_{X,\Omega} \bigg(\int_0^{C_1} \log\bigg( 1 +\frac{C_{14}}{h^{\nu}_{X,\Omega}t}\bigg) dt\bigg)^{1/2}\nonumber\\
     = & C_{15}C_1^{1/2}h_{X,\Omega}^{\nu_0}\rho^{\nu-\nu_0}_{X,\Omega} \bigg( C_{14}h^{-\nu}_{X,\Omega}\log\bigg(1 + \frac{C_1h^{\nu}_{X,\Omega}}{C_{14}}\bigg) + C_1\log\bigg(1+\frac{C_{14}}{C_1h^{\nu}_{X,\Omega}}\bigg) \bigg)^{1/2}.
\end{align}
By the basic inequality $\log(1+x)\leq x$, we conclude that
\begin{align*}
    C_{14}h^{-\nu}_{X,\Omega}\log\bigg(1 + \frac{C_1h^{\nu}_{X,\Omega}}{C_{14}}\bigg) \leq C_1.
\end{align*}
Consequently, by incorporating the condition $h_{X,\Omega}\leq 1$, we get
\begin{align}\label{ieqintpart4}
    & C_{14}h^{-\nu}_{X,\Omega}\log\bigg(1 + \frac{C_1h^{\nu}_{X,\Omega}}{C_{14}}\bigg) + C_1\log\bigg(1+\frac{C_{14}}{C_1h^{\nu}_{X,\Omega}}\bigg)\nonumber\\
    \leq & C_1 + C_1\log\bigg(1+\frac{C_{14}}{C_1h^{\nu}_{X,\Omega}}\bigg) \leq C_{16}\log\bigg(1+\frac{C_{17}}{h_{X,\Omega}}\bigg),
\end{align}
where in the last equality, we utilize $1 + b^a \leq (1+C'b)^{a+1}$ again. Combining \eqref{ieqintpart1}-\eqref{ieqintpart4}, we have shown that
\begin{align*}
    \mathbb{E}\sup_{x\in\Omega}|g_Z(x)| \leq C_{18}h_{X,\Omega}^{\nu_0}\rho^{\nu-\nu_0}_{X,\Omega}\sqrt{\log\bigg(1+\frac{C_{17}}{h_{X,\Omega}}\bigg)}.
\end{align*}

By Lemma \ref{th:Q}, we have the following upper bound for the maximum variance
$$\sigma^2_{g_Z}=\max_{x\in\Omega}Q^2(x)\leq C h^{2\nu_0}_{X,\Omega}\rho^{2(\nu-\nu_0)}_{X,\Omega}. $$
Then we complete the proof of Theorem \ref{Th:main} by invoking (\ref{concentrationGP1}) of Lemma \ref{concentrationGP}.

\subsubsection{Proof of Theorem \ref{th:low}}

According to Lemma \ref{lbexgau}, the key is to find a lower bound of $N(\epsilon,\Omega,d_{g_Z})$. The idea is as follows. Suppose for any $n$-point set $\{y_1,\ldots,y_n\}\subset \Omega$, we can find $y_0\in\Omega$ such that $\min_{1\leq j\leq n}d_{g_Z}(y_0,y_j)\geq \epsilon_n$ for some number $\epsilon_n>0$. Then $\Omega$ can not be covered by $n$ $(\epsilon_n,d_{g_Z})$-balls, and thus $N(\epsilon_n,\Omega,d_{g_Z})\geq n$.

Now take an arbitrary $n$-point set $Y=\{y_1,\ldots,y_n\}\subset\Omega$. For each $y_j$,
\begin{align*}
    d^2_{g_Z}(y,y_j) = & \mathbb{E}(Z( y)-\mathcal{I}_{\Phi,X}Z(y) - Z( y_j)+\mathcal{I}_{\Phi,X}Z(y_j)))^2 \nonumber\\
    \geq & \mathbb{E}(Z( y)-\mathcal{I}_{\Psi,X\cup Y}Z(y))^2,
\end{align*}
because $\mathcal{I}_{\Psi,X\cup Y}Z(y)$ is the best linear predictor of $Z(y)$ given $Z(x_1),\ldots,Z(x_n),Z(y_1),\ldots,$ $Z(y_n)$, and $\mathcal{I}_{\Phi,X}Z(y) + Z( y_j)-\mathcal{I}_{\Phi,X}Z(y_j)$ is a linear predictor and thus should have a greater mean squared prediction error. Corollary \ref{coro:Qlower} implies that
$$\sup_{y\in\Omega}\mathbb{E}(Z( y)-\mathcal{I}_{\Psi,X\cup Y}Z( y))^2=\sup_{y\in\Omega}P^2_{\Psi,X\cup Y}(y)\geq C^2_1(2n)^{-\frac{2\nu_0}{d}}.$$ Therefore, there exists $y_0\in\Omega$ such that $d^2_{g_Z}(y_0,y_j) \geq C^2_1(2n)^{-\frac{2\nu_0}{d}}/4$ for each $y_j$, which implies $N(C_1(2n)^{-\frac{\nu_0}{d}}/2,\Omega,d_{g_Z})\geq n$.
Now we invoke Lemma \ref{lbexgau} with $\eta=C_1(2n)^{-\frac{\nu_0}{d}}/2$ to obtain that
\begin{align}\label{lbofmmeasure2}
\mathbb{E} \sup_{ x\in \Omega} g_Z(x) \geq C_2 n^{-\frac{\nu_0}{d}} \sqrt{\log n}.
\end{align}
The desired result then follows from (\ref{concentrationGP2}) of Lemma \ref{concentrationGP}.

\subsection{$L_p$ results with $1\leq p<\infty$}
\label{sec:proofLp}
Our results for the $L_p$ norms with $1\leq p<\infty$ replies on a counterpart of the Borell-TIS inequality (Lemma \ref{concentrationGP}) under the $L_p$ norms. Such a result is given by Lemma \ref{thm211}; its proof is presented in Section \ref{Sec:proof211}.

\begin{lemma}\label{thm211}
Suppose $\Omega$ satisfies Condition \ref{C4}. Let $G$ be a zero-mean Gaussian process on $\Omega\subset\mathbb{R}^d$ with continuous sample paths almost surely and with a finite maximum pointwise variance $\sigma^2_G = \sup_{x\in \Omega} \mathbb{E}G(x)^2<\infty$. Then for all $u>0$ and $1\leq p < \infty$, we have
\begin{align*}
    & \mathbb{P}\left(\|G\|_{L_p(\Omega)} - \mathbb{E} \|G\|_{L_p(\Omega)} > u\right) \leq e^{-u^2/(2C_p\sigma^2_G)},\\
    & \mathbb{P}\left(\|G\|_{L_p(\Omega)} -\mathbb{E} \|G\|_{L_p(\Omega)}  <  - u\right) \leq e^{-u^2/(2C_p\sigma^2_G)},
\end{align*}
with $C_p={\rm Vol}(\Omega)^{2/p}$. Here ${\rm Vol}(\Omega)$ denotes the volume of $\Omega$.
\end{lemma}

\begin{remark}
Similar to the Borell-TIS inequality (Lemma \ref{concentrationGP}), the variation of $L_p$ norm of $G$ in Lemma \ref{thm211} can be controlled by its pointwise fluctuations. In fact, by letting $p\rightarrow\infty$, Lemma \ref{thm211} becomes Lemma \ref{concentrationGP}.
\end{remark}

As before, let $g_Z(x) = Z( x)-\mathcal{I}_{\Phi,X}Z( x)$, which is still a zero-mean Gaussian process; let $\sigma_{g_Z}^2 = \sup_{x\in \Omega}\mathbb{E}g_Z(x)^2$. In view of Lemma \ref{thm211}, the remaining task is to bound $\mathbb{E}\|g_Z\|_{L_p(\Omega)}$. This will be done by employing the known bounds of $\mathbb{E}\|g_Z\|^2_{L_2(\Omega)}$, as in Lemmas \ref{th:Q} and \ref{lem:Qlower}, together with Jensen's inequality and some other basic inequalities.

\subsubsection{Proof of Lemma \ref{thm211}}\label{Sec:proof211}
We will use the Gaussian concentration inequality given by Lemma \ref{conlips}. Its proof can be found in \citet[Lemma 2.1.6]{adler2009random}.
We say that $L$ is a Lipchitz constant of the function $h:\mathbb{R}^k \rightarrow \mathbb{R}$, if $|h(x)-h(y)|\leq L\|x-y\|$ for all $x,y\in\mathbb{R}^k$.
\begin{lemma}[Gaussian concentration inequality]\label{conlips}
Let $G_k$ be a k-dimensional vector of centered, unit-variance, independent Gaussian variables. If $h:\mathbb{R}^k \rightarrow \mathbb{R}$ has Lipschitz constant $L$, then for all $u>0$.
\begin{align*}
    \mathbb{P}(h(G_k) - \mathbb{E}h(G_k)>u)\leq e^{-u^2/(2L^2)}.
\end{align*}
\end{lemma}

The proof proceeds by approximating of the integral $\|G\|^p_{L_p(\Omega)}=\int_\Omega G(x)^p d x$ by a Riemann sum. For each $n=1,2,\ldots$, let $\{\Omega_{nj}\}_{j=1}^n$ be a partition of $\Omega$ such that
\begin{align}\label{thm211partition}
\max_{1\leq j\leq n} \text{Diam}(\Omega_{nj}) \rightarrow 0, \text{ as } n\rightarrow\infty,
\end{align}
where $\text{Diam}(\Omega_{nj})$ denotes the (Euclidean) diameter of $\Omega_{nj}$. We have $\Omega = \cup_{j=1}^n \Omega_{nj}$ and $\sum_{j=1}^n \text{Vol}(\Omega_{nj}) = \text{Vol}(\Omega)$. Let $w_{nj}=\text{Vol}(\Omega_{nj})$ and define
\begin{align*}
    \|a\|_w = \bigg(\sum_{j=1}^n w_{nj}|a_j|^p\bigg)^{1/p}
\end{align*}
for a vector $a=(a_1,...,a_n)^T$. Let $G_n=(G_{n1},...,G_{nn})^T$ with $G_{nj} = G(x_{nj})$ for some $x_{nj} \in \Omega_{nj}$. Therefore, $\|G_n\|_w$ is an approximate of $\|G\|_{L_p(\Omega)}$.

We first prove a similar result for $G_n$, and then arrive at the desired results by letting $n\rightarrow \infty$. Let $K$ be the $n\times n$ covariance matrix of $G_n$ on $\Omega$, with components $K_{ij} = \mathbb{E}(G_{ni}G_{nj})$. Define  $\sigma_{\Omega_n}^2 = \max_{1\leq j\leq n} \mathbb{E}G_{nj}^2$. Let $W$ be a vector of independent, standard Gaussian variables, and $A$ be a matrix such that $A^TA = K$. Thus $G_n$ has the same distribution as $AW$.

Consider the function $h(x) = \|Ax\|_w$. Let $e_j$ be the vector with one in the $j$th entry and zeros in other entries. Denote the $j$th entry of a vector $v$ by $[v]_j$. Then we have
\begin{align*}
    |h(x) - h(y)| = & |\|Ax\|_w - \|Ay\|_w|  \leq \|A(x - y)\|_w\\
   = & \bigg(\sum_{j=1}^n w_{nj}\left|[A(x-y)]_j\right|^p \bigg)^{1/p} =  \bigg(\sum_{j=1}^n w_{nj}|e_j^T A(x-y)|^p \bigg)^{1/p}\\
   \leq & \bigg(\sum_{j=1}^n w_{nj}\|e_j^T A\|^p\|x-y\|^p \bigg)^{1/p},
\end{align*}
where the first inequality follows from the triangle inequality (i.e., the Minkowski inequality); and the last inequality follows from the Cauchy-Schwarz inequality. Noting that for each $j$,
\begin{align*}
    \|e_j^T A\|^2 = e_j^TA^TAe_j = \mathbb{E}(G_{nj}^2)\leq \sigma^2_{\Omega_n},
\end{align*}
we have
\begin{align*}
    |h(x) - h(y)| \leq & \text{Vol}(\Omega)^{1/p}\sigma_{\Omega_n}\|x-y\|,
\end{align*}
which implies $h$ is a Lipschitz continuous function with Lipschitz constant $\text{Vol}(\Omega)^{1/p}\sigma_{\Omega_n}$. Because $G_n$ and $AW$ have the same distribution, and together with Lemma \ref{conlips}, we obtain
\begin{align}\label{thm211neq1}
    \mathbb{P}\left(\|G_n\|_w - \mathbb{E} \|G_n\|_w > u\right) \leq e^{-u^2/(2C_p\sigma^2_{\Omega_n})},
\end{align}
where $C_p=\text{Vol}(\Omega)^{2/p}$. Similarly, by considering $h(x) = -\|Ax\|_w$, we can obtain
\begin{align}\label{thm211neq2}
    \mathbb{P}\left(\|G_n\|_w -\mathbb{E} \|G_n\|_w  <  - u\right) \leq e^{-u^2/(2C_p\sigma^2_{\Omega_n})}.
\end{align}
To prove the desired results, we let $n\rightarrow\infty$ in (\ref{thm211neq1}) and (\ref{thm211neq2}). First we show that the left-hand sides of (\ref{thm211neq1}) and (\ref{thm211neq2}) tend to $\mathbb{P}\left(\|G\|_{L_p(\Omega)} - \mathbb{E} \|G\|_{L_p(\Omega)} > u\right)$ and \newline $\mathbb{P}\left(\|G\|_{L_p(\Omega)} - \mathbb{E} \|G\|_{L_p(\Omega)}< -u\right)$, respectively. According to Lebesgue's dominated convergence theorem, it suffices to prove that
\begin{align}\label{thm211as}
\|G_n\|_w\rightarrow \|G\|_{L_p(\Omega)},~~ a.s., \text{ as } n\rightarrow\infty,
\end{align}
and
\begin{align}\label{thm211mean}
\mathbb{E}\|G_n\|_w\rightarrow \mathbb{E}\|G\|_{L_p(\Omega)},\text{ as } n\rightarrow\infty,
\end{align}
as the indicator function is dominated by one. Since $G$ has continuous sample paths with probability one, (\ref{thm211as}) is an immediate consequence of the convergence of Riemann integrals. Now we prove (\ref{thm211mean}).
Note that $\|G_n\|_w \leq \text{Vol}(\Omega)^{1/p}\sup_{x\in \Omega}|G(x)|$ and Lemma \ref{concentrationGP} suggests that $\mathbb{E}\sup_{x\in \Omega}|G(x)|<\infty$. Thus Lebesgue's dominated convergence theorem implies (\ref{thm211mean}). Now we consider the right-hand sides of (\ref{thm211neq1}) and (\ref{thm211neq2}). To prove the desired results, it remains to prove $\sigma^2_{\Omega_n}\rightarrow\sigma^2_G$. By the definition of $\sigma^2_{\Omega_n}$, we have
\begin{align}\label{thm211sigma}
\max_{1\leq j\leq n}\inf_{x\in\Omega_{nj}}\mathbb{E}G(x)^2\leq \sigma^2_{\Omega_n}\leq \sigma^2_G.
\end{align}
The almost sure continuity of $G$ implies that $\mathbb{E}G(x)^2$ is continuous in $x$. Since $\Omega$ is compact, $\mathbb{E}G(x)^2$ is also uniformly continuous. Therefore, the condition of the partitions in (\ref{thm211partition}) implies
\begin{align}\label{thm211inf}
\max_{1\leq j\leq n}\inf_{x\in\Omega_{nj}}\mathbb{E}G^2(x)\rightarrow \sigma^2_G, \text{ as } n\rightarrow\infty.
\end{align}
Combining (\ref{thm211sigma}) and (\ref{thm211inf}) proves $\sigma^2_{\Omega_n}\rightarrow\sigma^2_G$, which completes the proof.

\subsubsection{Proof of Theorem \ref{Th:wtwLp}}
By Fubini's theorem,
\begin{align}\label{upElp}
    \mathbb{E} \|Z-\mathcal{I}_{\Phi,X}Z\|_{L_p(\Omega)}^p & = \int_{x\in \Omega}\mathbb{E}|Z(x)-\mathcal{I}_{\Phi,X}Z(x)|^p dx\nonumber\\
    & = \int_{x\in \Omega} \frac{2^{p/2}\Gamma(\frac{p+1}{2})}{\sqrt{\pi}}\big(\mathbb{E}(Z(x)-\mathcal{I}_{\Phi,X}Z(x))^2\big)^{p/2} dx\nonumber\\
    & \leq C_1 \sigma_{g_Z}^p.
\end{align}
The second equality of \eqref{upElp} is true because $Z(x)-\mathcal{I}_{\Phi,X}Z(x)$ follows a normal distribution with mean zero, and the absolute moments of a normal random variable $X_\sigma\sim N(0,\sigma^2)$ can be expressed by its variance as
\begin{eqnarray}\label{Gaussianmoments}
\mathbb{E}|X_\sigma|^p=\sigma^p \cdot \frac{2^{p/2}\Gamma\left(\frac{p+1}{2}\right)}{\sqrt{\pi}};
\end{eqnarray}
see \cite{walck1996hand}.
By combining Lemma \ref{thm211} and \eqref{upElp}, we have
\begin{align}\label{Lpineqpf1}
    e^{-u^2/2C_p\sigma^2_{g_Z}} & \geq \mathbb{P}\left(\|g_Z\|_{L_p(\Omega)}  > \mathbb{E} \|g_Z\|_{L_p(\Omega)} + u\right) \nonumber\\
    & \geq \mathbb{P}\left(\|g_Z\|_{L_p(\Omega)}^p  > 2^{p-1}(\mathbb{E} \|g_Z\|_{L_p(\Omega)}^p + u^p)\right) \nonumber\\
    & \geq \mathbb{P}\left(\|g_Z\|_{L_p(\Omega)}^p  > 2^{p-1}(C_1 \sigma_{g_Z}^p + u^p)\right) \nonumber\\
    & \geq \mathbb{P}\left(\|g_Z\|_{L_p(\Omega)}  > 2^{1-1/p}(C_1^{1/p} \sigma_{g_Z} + u)\right),
\end{align}
where the second inequality follows from the Jensen's inequality and the $c_r$-inequality. Combining Lemma \ref{th:Q} and \eqref{Lpineqpf1} completes the proof.

\subsubsection{Proof of Theorem \ref{Th:mainLp}}

The proof of Theorem \ref{Th:mainLp} is similar to that of Theorem \ref{Th:wtwLp}. The only difference here is that at the last step we employ Lemma \ref{th:Qunder} instead of Lemma \ref{th:Q}.

\subsubsection{Proof of Theorem \ref{coro:lowLp}}

Take a quasi-uniform design $X'\subset\Omega$ with card$(X')=n$. Obviously $h_{X\cup X',\Omega}\leq h_{X',\Omega}$. By Proposition 14.1 of \cite{wendland2004scattered}, $h_{X',\Omega} \leq C n^{-1/d}$. By H\"older's inequality, we have $\|f\|_{L_2(\Omega)}\leq \|f\|_{L_1(\Omega)}^{1/4}\|f\|_{L_3(\Omega)}^{3/4}$ for any continuous function $f$, which implies
\begin{align}\label{pl2inter}
 & \bigg(\int_{x\in \Omega}\mathbb{E}(Z(x)-\mathcal{I}_{\Psi,X\cup X'}Z(x))^2 dx\bigg)^{1/2}\nonumber\\
 \leq & \bigg(\int_{x\in \Omega}\big(\mathbb{E}(Z(x)-\mathcal{I}_{\Psi,X\cup X'}Z(x))^2\big)^{1/2} dx\bigg)^{1/4} \bigg(\int_{x\in \Omega}\big(\mathbb{E}(Z(x)-\mathcal{I}_{\Psi,X\cup X'}Z(x))^2\big)^{3/2} dx\bigg)^{1/4}.
\end{align}
Applying Lemma \ref{th:Qunder} to $\sup_{x\in \Omega}\mathbb{E}(Z(x)-\mathcal{I}_{\Psi,X\cup X'}Z(x))^2$ with $\nu = \nu_0$ yields
\begin{align}\label{ub3ininter}
    & \bigg(\int_{x\in \Omega}\big(\mathbb{E}(Z(x)-\mathcal{I}_{\Psi,X\cup X'}Z(x))^2\big)^{3/2} dx\bigg)^{1/4} \nonumber\\
    \leq & C_1\bigg(\sup_{x\in \Omega} \big(\mathbb{E}(Z(x)-\mathcal{I}_{\Psi,X\cup X'}Z(x))^2\big)^{3/2}\bigg)^{1/4}\nonumber\\
    \leq & C_2 h_{X\cup X',\Omega}^{\frac{3\nu_0}{4}}
    \leq C_2 h_{ X',\Omega}^{\frac{3\nu_0}{4}} \leq C_3 n^{-\frac{3\nu_0}{4d}}.
\end{align}
The left hand side of \eqref{pl2inter} can be bounded from below by using Lemma \ref{lem:Qlower}, which yields
\begin{align}\label{lb2ininter}
    \bigg(\int_{x\in \Omega}\mathbb{E}(Z(x)-\mathcal{I}_{\Psi,X\cup X'}Z(x))^2 dx\bigg)^{1/2}&=\left(\mathbb{E}\|Z-\mathcal{I}_{\Psi,X\cup X'}Z\|^2_{L_2(\Omega)}\right)^{1/2}\nonumber\\
    = \|P_{\Psi,X\cup X'}\|_{L_2(\Omega)} \geq C_4 (2n)^{-\nu_0/d},
\end{align}
where the equality follows from Fubini's theorem.
Plugging \eqref{ub3ininter} and \eqref{lb2ininter} into \eqref{pl2inter}, we have
\begin{align}\label{lbpininter}
    \int_{x\in \Omega}\big(\mathbb{E}(Z(x)-\mathcal{I}_{\Psi,X\cup X'}Z(x))^2\big)^{1/2} dx \geq C_5n^{\frac{3\nu_0}{d}} n^{-\frac{4\nu_0}{d}} = C_5n^{-\frac{\nu_0}{d}}.
\end{align}
By Fubini's theorem and \eqref{lbpininter}, it can be seen that
\begin{align}\label{lbexppsmall}
    \mathbb{E} \|Z-\mathcal{I}_{\Phi,X}Z\|_{L_1(\Omega)} & = \int_{x\in \Omega}\mathbb{E}|Z(x)-\mathcal{I}_{\Phi,X}Z(x)| dx\nonumber\\
    & = \int_{x\in \Omega} \frac{2^{1/2}}{\sqrt{\pi}}\big(\mathbb{E}(Z(x)-\mathcal{I}_{\Phi,X}Z(x))^2\big)^{1/2} dx\nonumber\\
    & \geq C_6\int_{x\in \Omega}\big(\mathbb{E}(Z(x)-\mathcal{I}_{\Psi,X\cup X'}Z(x))^2\big)^{1/2} dx\nonumber\\
    & \geq C_7 n^{-\frac{\nu_0}{d}},
\end{align}
where the second equality follows from (\ref{Gaussianmoments}) with $p=1$; the first inequality is because $\mathcal{I}_{\Psi,X\cup X'}Z(x)$ is the best linear predictor of $Z(x)$.
For $1\leq p <\infty$ and any $u>0$, applying Lemma \ref{thm211} yields
\begin{align}\label{pflblpprob}
    e^{-u^2/(2C_8\sigma^2_\Omega)} & \geq \mathbb{P}\left(\|g_Z\|_{L_p(\Omega)}  < \mathbb{E} \|g_Z\|_{L_p(\Omega)} - u\right) \nonumber\\
    & \geq \mathbb{P}\left(\|g_Z\|_{L_p(\Omega)}^p  < 2^{1-p}(\mathbb{E} \|g_Z\|_{L_p(\Omega)})^p - u^p\right) \nonumber\\
    & \geq \mathbb{P}\left(\|g_Z\|_{L_p(\Omega)}^p  < 2^{1-p}(C_9\mathbb{E} \|g_Z\|_{L_1(\Omega)})^p - u^p\right) \nonumber\\
    & \geq \mathbb{P}\left(\|g_Z\|_{L_p(\Omega)}^p  < C_{10} n^{-\nu_0p/d} - u^p\right) \nonumber\\
    & \geq \mathbb{P}\left(\|g_Z\|_{L_p(\Omega)}  < C_{11} n^{-\nu_0/d} - u\right).
\end{align}
In \eqref{pflblpprob}, the second inequality is because of Jensen's inequality; the third inequality is because of the fact $\|g_Z\|_{L_p(\Omega)} \geq C_9\|g_Z\|_{L_1(\Omega)}$ for some constant $C_9>0$ depending on $p$ and $\Omega$; the fourth inequality is by \eqref{lbexppsmall}; and the last inequality is true because of the elementary inequality $(a + b)^p \geq a^p + b^p$ for $a,b >0$. Thus, we finish the proof of Theorem \ref{coro:lowLp}.

\section*{Acknowledgements}
The authors are grateful to the AE and three reviewers for very helpful comments and suggestions.

\appendix
\section{Distributions and Asymptotic Orders in Example \ref{examplerdsample}}\label{app:distribution}
\begin{proposition}\label{propexpunif}
Let $x_1,\ldots,x_n$ be mutually independent random variables following the uniform distribution on $[0,1]$. Denote their order statistics as $$0=x_{(0)}\leq x_{(1)}\leq\cdots\leq x_{(n)}\leq x_{(n+1)}=1.$$ Let $y_1,\ldots,y_n$ be mutually independent random variables following the exponential distribution with mean one. Therefore, $(x_{(1)},\ldots,x_{(n)})$ has the same distribution as
$$\left(\frac{y_1}{\sum_{j=1}^{n+1} y_j},\ldots,\frac{\sum_{j=1}^n y_j}{\sum_{j=1}^{n+1} y_j}\right). $$
\end{proposition}
The proof of Proposition \ref{propexpunif} relies on the following lemma.
\begin{lemma}[Lemma 4.5.1 of \cite{resnick1992adventures}]\label{lemmaasp}
Let $y_1,\ldots,y_n,y_{n+1}$ be mutually independent random variables following the exponential distribution with mean one. Define $E_k = \sum_{i=1}^k y_i$ for $k=1,...,n+1$. Then conditional on $E_{n+1} = t$, the joint density of $E_1,...,E_{n}$ is
\begin{align*}
    f_{E_1,...,E_{n}|E_{n+1}=t}(u_1,...,u_{n}) = \left\{\begin{array}{ll}
        \frac{n!}{t^{n}}, & \mbox{ if }0<u_1<\cdots<u_n<t,  \\
         0, & \mbox{ otherwise.}
    \end{array}\right.
\end{align*}
\end{lemma}
\begin{proof}[Proof of Proposition \ref{propexpunif}]
By Lemma \ref{lemmaasp}, it can be shown that
\begin{align*}
    f_{\frac{E_1}{E_{n+1}},...,\frac{E_{n}}{E_{n+1}}|E_{n+1}=t}(u_1,...,u_{n}) = f_{E_1,...,E_{n}|E_{n+1}=t}(u_1t,...,u_{n}t)\\
    = \left\{\begin{array}{ll}
        n!, & \mbox{ if }0<u_1<\cdots<u_{n}<1,  \\
         0, & \mbox{ otherwise.}
    \end{array}\right.
\end{align*}
which implies
\begin{align}\label{eqdisexp}
    f_{\frac{E_1}{E_{n+1}},...,\frac{E_{n}}{E_{n+1}}}(u_1,...,u_{n}) = \left\{\begin{array}{ll}
        n!, & \mbox{ if }0<u_1<...<u_{n}<1,  \\
         0, & \mbox{ otherwise,}
    \end{array}\right.
\end{align}
by taking the expectation with respect to $E_{n+1}$. Note that \eqref{eqdisexp} is the same as the joint density of order statistics $(x_{(1)},\ldots,x_{(n)})$, which completes the proof.
\end{proof}
\begin{proposition}\label{propexpfacts}
Let $y_1,\ldots,y_n$ be mutually independent random variables following the exponential distribution with mean one. Then $\max y_j \asymp \log n$, $\min y_j \asymp 1/n$, $\max y_j/\min y_j \asymp n\log n$ and $\max y_j/\sum_{j=1}^{n}y_j = O_{\mathbb{P}}(n^{-1}\log n)$.
\end{proposition}
\begin{remark}
For positive sequences of random variables $a_n$, $b_n$, we write $a_n\asymp b_n$ if $a_n = O_{\mathbb{P}}(b_n)$ and $b_n = O_{\mathbb{P}}(a_n)$.
\end{remark}
\begin{proof}[Proof of Proposition \ref{propexpfacts}]
We first show that for any $\epsilon>0$, there exists an $M$ and $N$ such that
\begin{align}
    \sup_{n>N}\mathbb{P}\left(\frac{\max y_j}{\log n}>M\right)\leq \epsilon,& \sup_{n>N}\mathbb{P}\left(\frac{\log n}{\max y_j}>M\right)\leq \epsilon,\label{inexpmaxyj}\\
    \sup_{n>N}\mathbb{P}\left(n\min y_j>M\right)\leq \epsilon,& \sup_{n>N}\mathbb{P}\left(\frac{1}{n\min y_j}>M\right)\leq \epsilon.\label{inexpminyj}
\end{align}
For any $x>0$, it can be checked that
\begin{align*}
    \mathbb{P}(\max y_j \leq x) = (1 - e^{-x})^n,
\end{align*}
which, for $n>N$, by  Bernoulli's inequality and the basic inequality $\log(1+x)<x$, implies
\begin{align*}
    \mathbb{P}(\max y_j \geq M\log n) = & 1 - (1 - n^{-M})^n\leq n^{-M+1}\leq N^{-M+1}\rightarrow 0,\\ \mathbb{P}\left(\max y_j \leq \frac{\log n}{M}\right) = & (1 - n^{-\frac{1}{M}})^n = e^{n\log(1 - n^{-\frac{1}{M}})}\leq e^{-n^{1-\frac{1}{M}}}\leq e^{-N^{1-\frac{1}{M}}}\rightarrow 0,
\end{align*}
as $N,M\rightarrow \infty$. This finishes the proof of \eqref{inexpmaxyj}.

For any $x>0$, $\min y_j$ has the cumulative distribution function
\begin{align*}
    \mathbb{P}(\min y_j \leq x) = 1 - (e^{-x})^n=1 - e^{-nx}.
\end{align*}
Therefore, we have
\begin{align*}
   \mathbb{P}\bigg(\min y_j \geq \frac{M}{n}\bigg) = e^{-M}, \mbox{ and } \mathbb{P}\bigg(\min y_j \leq \frac{1}{nM}\bigg) = 1- e^{-\frac{1}{M}} \rightarrow 0,
\end{align*}
as $M\rightarrow \infty$, which finishes the proof of \eqref{inexpminyj}. Note \eqref{inexpmaxyj} and \eqref{inexpminyj} imply $\max y_j \asymp \log n$ and $\min y_j \asymp 1/n$, respectively. Because for positive sequences $a_n,b_n,c_n$, $a_n=O_{\mathbb{P}}(b_n)$ and $b_n=O_{\mathbb{P}}(c_n)$ implies $a_n =O_{\mathbb{P}}(c_n)$, we have $\max y_j/\min y_j \asymp n\log n$.

Next we show $\max y_j/\sum_{j=1}^{n}y_j = O_{\mathbb{P}}(n^{-1}\log n)$. Because we have shown that $\max y_j \asymp \log n$, it suffices to show $n = O_{\mathbb{P}}(\sum_{j=1}^{n}y_j)$, which is equivalent to show that for any $\epsilon>0$, there exists an $M$ and $N$ such that
\begin{align}
    \sup_{n>N}\mathbb{P}\left(\sum_{j=1}^{n}y_j< \frac{n}{M}\right) \leq \epsilon. \label{inexpmaxsumyj}
\end{align}
By Chebyshev's inequality, for $n>N$
\begin{align*}
& \mathbb{P}\left(\sum_{j=1}^{n}y_j< \frac{n}{M}\right)  =
    \mathbb{P}\left(\frac{1}{n}\sum_{j=1}^{n}y_j - 1< \frac{1}{M} - 1\right)\\
    & \leq \mathbb{P}\left(\bigg|\frac{1}{n}\sum_{j=1}^{n}y_j - 1\bigg|>\bigg|\frac{1}{M} - 1\bigg|\right) \leq \frac{1}{n(1-1/M)^2}\leq \frac{1}{N(1-1/M)^2} \rightarrow 0,
\end{align*}
as $N,M\rightarrow \infty$. This shows $n = O_{\mathbb{P}}(\sum_{j=1}^{n}y_j)$, and finishes the proof.
\end{proof}

\bibliography{robustness}

\end{document}